\documentclass[12pt]{article}
\usepackage{graphicx}
\usepackage[utf8]{inputenc}
\usepackage[english]{babel}
\usepackage{tikz}
\usetikzlibrary{cd}
\usepackage{bbm}
\usepackage{amssymb, amsmath, color, amsthm}
\usepackage{hyperref}
\newtheorem{theorem}{Theorem}[section]
\newtheorem{definition}{Definition}[section]

\newtheorem{lemma}{Lemma}[section]
\newtheorem{prop}{Proposition}[section]
\newtheorem{property}{Property}[section]
\newtheorem{remark}{Remark}[section]
\newtheorem{example}{Example}[section]
\newtheorem*{theoremA}{Theorem A}
\newtheorem*{theoremB}{Theorem B}
\usepackage[top=1in, bottom=1in, right=1in, left=1in]{geometry}

\newcommand{\cT}{\mathcal{T}}

\newcommand{\mydotw}[1]{\begin{scope}[shift={#1}] \fill[shift only,white] (0,0) circle (1.5pt); \draw[shift only,thick]  (0,0) circle (1.5pt);   \end{scope}}

\begin{document}

\title{Module categories for $A_n$ web categories from $\tilde{A}_{n-1}$-buildings.}
\author{Emily McGovern \\  \href{ecmcgove@ncsu.edu}{ecmcgove@ncsu.edu} } 

\maketitle

\textbf{Abstract: }We equip the category of vector bundles over the vertices of a locally finite $\tilde{A}_{n-1}$ building $\Delta$ with the structure of a module category over a category of type $A_{n}$ webs in positive characteristic. This module category is a $q$-analogue of the $Rep(SL_{n})$ action on vector bundles over the $sl_n$ weight lattice.  We show our module categories are equivariant with respect to symmetries of the building, and when a group $G$ acts simply transitively on the vertices of $\Delta$ this recovers the fiber functors constructed by Jones.

\section{Introduction}
\label{introduction}

In \cite{Jones21}, the author presents the construction of a fiber functor on a certain monoidal category $Web(SL_{n}^{-})$ of type A webs in positive characteristic, which can be extended to a non-standard fiber functor on the $Rep(SL_{2k+1})$ as shown in \cite[Example 7.3.4]{CEOP21}. This functor is constructed through the use of a combinatorial structure called a triangle presentation of type $\tilde{A}_{n-1}$. Triangle presentations were first introduced in \cite{CMSZ93} as a characterization of a group acting simply transitively on a building of type $\tilde{A}_{n-1}$. The existence of a triangle presentation relies  on the presence of a vertex-transitive action of a group on a building. This type of action requires the high level of symmetry found only in type $A$ combinatorics. Thus we cannot apply the same ideas to obtain non-standard fiber functors in other types. 

Instead of looking for fiber functors, it is natural to look more generally for \textit{module categories} arising from the combinatorics of buildings, which may exist outside type A. There is a canonical module category for $\text{Rep}(SL_{n})$ whose underlying category is the category of vector bundles on the type $A$ weight lattice $L$, which arises from restriction. In a certain sense, a locally finite building of type $\tilde{A}_{n-1}$ and order $q$ (see note after Lemma \ref{buildingFPG}) can be thought of as a $q$-analogue of $L$. This motivates the following result, which is the main theorem of the paper.
\begin{theoremA}
\label{theoremA}
If $\Bbbk$ is a field of characteristic $p \geq n-1$ and $\Delta$ is an $\tilde{A}_{n-1}$ building of order $q \equiv 1 \mod p$. There is a monoidal functor $Web(SL_n^{-}) \rightarrow End(Vec(\Delta))$, where both categories are defined over $\Bbbk$. When $n$ is odd, this equips $Vec(\Delta)$ with the structure of a module category over $Tilt(SL_n)$.
\end{theoremA}

In the above theorem, $Vec(\Delta)$ denotes the category of vector bundles over the set of vertices of $\Delta$. From this module category, we can recover the non-standard fiber functor introduced in \cite{Jones21} by studying symmetries. Any action of a group $G$ on $\Delta$ induces an action on the category $Vec(\Delta)$ and so we can consider its equivariantization $Vec(\Delta)^{G}$. Our second main theorem extends the module category of Theorem A to $Vec(\Delta)^{G}$. 
\begin{theoremB}
If $\Bbbk$ is a field as in Theorem A, then for any type-rotating action of a group $G$ on $\Delta$ there exists a monoidal functor $Web(SL_n^{-}) \rightarrow End(Vec(\Delta)^{G})$, equipping $Vec(\Delta)^{G}$ with the structure of a $Web(SL_n^{-})$ module category. 
\end{theoremB} 

When $G$ acts simply transitively on an affine building in type $A$, then $Vec(\Delta)^{G} \simeq Vec$, and as we have $\text{End}(Vec)\cong Vec$, we recover the fiber functor on the $Web(SL_n^{-})$ category constructed in \cite{Jones21}.

In order to build the module category from Theorem $A$, we consider an intermediate tensor category $G(\Delta)$, called the \textit{graph planar algebra} of $\Delta$. Graph planar algebras were introduced in \cite{Jones98}, \cite{Jones99} and \cite{Mor10}. Our version varies slightly, but we show in remark \ref{gpa from paper} how it fits into the pre-existing framework.

We expect similar results relating web categories and buildings should be true outside of type $A$. Following the introduction of web categories in \cite{Kup96}, there have been many web categories defined in other types. For example, \cite{BERT21} gives a presentation of a web category that corresponds to the representation theory of $sp_{2n}$ and the Coxeter/Lie combinatorics in type $C$. However, it is not immediately clear which presentations of these web categories will yield natural module category structures. Finally, we remark that as in \cite{Jones21}, the functors and natural transformations defining our module category make sense in characteristic $0$, but do not satisfy the $Web(SL_n^{-})$ relations. Instead they generate a new category related to the quantum automorphism group of the building, introduced and studied in \cite{https://doi.org/10.48550/arxiv.2209.03770}, building on the previous work of \cite{MR3899967}. It would be interesting to precisely clarify the relationship between $\text{Web}(SL^{-}_{n})$ and the quantum automorphism group of these graphs.

 The structure of this paper is as follows. In section \ref{preliminaries}, we introduce the main players in our work, namely $\tilde{A}_{n-1}$ buildings and the categories $Web(SL_n^{-})$ and $G(\Delta)$ for a locally finite building $\Delta$ of type $\tilde{A}_{n-1}$. Section \ref{functorsection} defines the functor $Web(SL_n^{-}) \rightarrow G(\Delta)$, while also proving some results about buildings of order 1 (which are Coxeter complexes). Section \ref{Constructing Module Categories} then details the construction of $Vec(\Delta)^{G}$ as a $Web(SL_n^{-})$ module category and provide examples of the category for several actions of $G$ on $\Delta$.

\section{Preliminaries}
\label{preliminaries}
\subsection{Buildings and Graph Planar Algebras}
\label{buildings and gpas}
In this section, we will present the definitions of Coxeter complexes and buildings, with a focus on types $\tilde{A}_{n-1}$ and $A_{n-1}$. We will also state several facts about the structure of these objects and discuss their connections to each other and to the representation theory of the Lie algbera $sl_n$. We will then introduce graph planar algebras, and describe the construction of a graph planar algebra for any building of type $\tilde{A}_{n-1}$. We refer the reader to \cite{AB10} for futher background on Coxeter systems and simplical complexes. 
\subsubsection{Coxeter Complexes and Buildings}
\label{CCs and Buildings}
The special subgroups of a Coxeter system $(W,S)$ are the subgroups $\langle S' \rangle$ for some $S' \subseteq S$ and the special cosets are the cosets of theses subgroups. We can define the Coxeter complex of $(W,S)$, $\Sigma(W,S)$ or $\Sigma$ if context is clear,  as the poset of special cosets under the opposite of the inclusion relation. This poset is a simplicial complex (see \cite[Theorem 3.5]{AB10}). Moreover, it is labelable, which means that there is a map from $V(\Sigma)$ to some label set $I$ (for us $\{1, ..., |S|\}$) such that the vertices of every maximal simplex are in bijection with the elements of $I$ (note that this property is called colorable in \cite{AB10}).

Recall that we can represent a Coxeter system $(W,S)$ as a graph called the Coxeter diagram of $(W,S)$. In this paper, we are interested in the groups with the following Coxeter diagrams. 
\begin{equation}
\label{Coxeter Diagrams}
\begin{tikzpicture}[scale=1, baseline = -.1cm]
 \draw[fill=black] (-2,0) circle (1.5pt);
 \draw[fill=black] (-1,0) circle (1.5pt);
 \draw[fill=black] (0,0) circle (1.5pt);
 \draw[fill=black] (1,0) circle (1.5pt);
 \draw[fill=black](2,0) circle (1.5pt);
 \draw (-2,0) to (-1, 0);
 \draw (-1,0) to (0, 0);
 \draw[dotted] (0,0) to (1,0);
 \draw (1,0) to (2, 0);
 \node(none) at (-2,0.5){1};
 \node(none) at (-1,0.5){2};
 \node(none) at (0,0.5){3};
 \node(none) at (1,0.5){n-2};
 \node(none) at (2,0.5){n-1};
\end{tikzpicture}\,\,\,\,\,\,\,\,\,\,\,\,
\begin{tikzpicture}[scale=1, baseline = -.1cm]
 \draw[fill=black] (-2,0) circle (1.5pt);
 \draw[fill=black] (-1,0) circle (1.5pt);
 \draw[fill=black] (0,0) circle (1.5pt);
 \draw[fill=black] (1,0) circle (1.5pt);
 \draw[fill=black](2,0) circle (1.5pt);
 \draw[fill=black](0,-1) circle (1.5pt);
 \draw (-2,0) to (-1, 0);
 \draw (-1,0) to (0, 0);
 \draw[dotted] (0,0) to (1,0);
 \draw (1,0) to (2, 0);
 \draw (-2,0) to (0,-1);
 \draw (2,0) to (0, -1);
 \node(none) at (-2,0.5){1};
 \node(none) at (-1,0.5){2};
 \node(none) at (0,0.5){3};
 \node(none) at (1,0.5){n-2};
 \node(none) at (2,0.5){n-1};
 \node(none) at (0,-1.5){n};
\end{tikzpicture}
\end{equation}
These diagrams represent the Coxeter systems of type $A_{n-1}$ and  $\tilde{A}_{n-1}$ respectively. A group of type $A_{n-1}$ is isomorphic to the symmetric group $S_n$ and  $\tilde{A}_{n-1}$ is the corresponding affine, irreducible system. 

Also, recall that two simplicies in a simplical complex $\Sigma$ are joinable, if they have an upper bound in $\Sigma$. The link of a simplex $X$  in a simplicial complex $\Sigma$ (denoted $lk_{\Sigma}(X)$) is  the subcomplex of all simplices that are both disjoint from and joinable to $X$ (see \cite[Definition A.19]{AB10}). If $a$ is a vertex of $\Sigma$, then $lk_{\Sigma}(a)$ is the induced subcomplex with vertex set $V = \{b | b \text{ is connected by an edge to } a \text{ in } \Sigma\}$. If $lk_\Sigma(a)$ has finitely many vertices for every vertex $a \in \Sigma$, we call $\Sigma$ locally finite. It is well known that the link of any vertex in the $\tilde{A}_{n-1}$ Coxeter complex is isomorphic to the $A_{n-1}$ Coxeter complex. Additionally, the $A_{n-1}$ Coxeter complex is isomorphic to the flag complex of proper, non-empty subsets of $\{1,...,n\}$ with incidence defined by inclusion (the flag complex is the simplicial complex with these subsets as the vertices, and finite flags as simplices). This isomorphism is consistent with the canonical labelling of $\Sigma$ (i.e. vertices with the same label are matched to subsets of the same size). We will use these facts throughout this paper.

The $\tilde{A}_{n-1}$ Coxeter complex $\Sigma$ appears in the representation theory of the Lie algebra $sl_n$. In particular, the $sl_n$ coroot lattice is embedded in $\Sigma$ (see \cite[Section 10.1.8]{AB10}.  If $n = 3$, it is easy to see that as graphs, $\Sigma$ is isomorphic as a graph to the $sl_n$ weight lattice. We show in section \ref{degernate case section} the connection between the 1-skeleton of the $\tilde{A}_{n-1}$ Coxeter complex, the $sl_n$ weight lattice, and the Cayley graph of a distinguished group that we will define in that section.  A way to see this is by choosing a Weyl chamber as in \cite[Lecture 14]{FH04}. The collection of these Weyl chambers is isomorphic to a $A_{n-1}$ Coxeter complex (or more specifically, the flag complex of proper subsets of $\{1,...,n\}$ as previously defined). 

A building of type $\tilde{A}_{n-1}$ is a simplicial complex which is built out of Coxeter complexes, originally introduced by Tits \cite[Definiton 3.1]{Tits74}. We borrow the following definition from \cite[Defintion 4.1]{AB10}.
\begin{definition}
\label{buidling}
A building $\Delta$ is a simplicial complex  with a distinguished set of subcomplexes called apartments that satisfy the following properties:
\begin{enumerate}
    \item Each apartment is a Coxeter complex for some Coxeter system $(W,S)$.
    \item If $A$ and $B$ are simplices in $\Delta$, then there is some apartment $\Sigma_{AB}$ containing them both.  
    \item There is an isomorphism between any two apartments $\Sigma$ and $\Sigma'$ fixing the intersection $\Sigma \cap \Sigma'$. 
\end{enumerate}

\end{definition}
We call the collection $\mathcal{A}$ of apartments of $\Delta$ a system of apartments, and note that one building can be equipped with multiple possible systems of apartments. All apartments of a building must be Coxeter complexes of the same type (see \cite[Proposition 4.7]{AB10}). So we will say $\Delta$ is type $\tilde{A}_{n-1}$ if its apartments are of type $\tilde{A}_{n-1}$.  
\begin{example} \label{BuidlingExample}
\normalfont
(see \cite[Section 6.9]{AB10}) 

The canonical examples are the Bruhat-Tits buildings. Suppose $K$ is a field with a discrete valuation (e.g. The rational numbers with the $p$-adic valuation for some prime $p$, or its completion, the field of $p$-adic numbers), and $A$ is its valuation ring (i.e. $A = \{x \in K | v(X) \geq 0\}$). Recall that we can choose an element $\pi \in A$ with $v(\pi) = 1$ so that every $x \in K^{*}$ is of the form $\pi^{n}u$ for some $u \in A^{*}$ (Here  the valuation is normalized, so that it only takes integer values).

Now, consider the collection of $A$-lattices in $K^{n}$ (i.e, $Av_1 \oplus ... \oplus Av_n$ for some basis $v_1, ..., v_n$ of $K^{n}$). We can define an equivalence relation on lattices $L$ by saying that $L \equiv L'$ if and only if $L = \lambda L'$ for some scalar $\lambda \in K$. We can take the equivalence classes of this relation as the vertices of a graph. 

We define incidence in this graph in the following manner. Take two equivalence classes $\Lambda$ and $\Lambda'$ to be incident if there exist representatives $L \in \Lambda$ and $L' \in \Lambda'$ such that $\pi L \subset L' \subset L$ where $\pi$ is the distinguished element of $A$ we chose before. The flag complex of this graph is  a building of type $\tilde{A}_{n-1}$. 
\end{example} 
\begin{remark}
\normalfont
\label{BNpairs}
The construction of the previous buildings come from the theory of $BN$-pairs. Briefly, this method of creating buildings takes advantage of a group $G$ with subgroups $B$ and $N$ with a particular set of properties. So buildings discovered in this manner have a natural  transitive action of the group $G$ on the vertices. A more thorough treatment of $BN$-Pairs can be found in \cite[Section 6.2]{AB10}
\end{remark}

The labeling on each apartment $\Sigma$ previously discussed can be extended to a consistent labeling of the building $\Delta$ (i.e. a labeling that restricts to the labeling of the Coxeter complex on any apartment $\Sigma$). So a building is itself a labelable  chamber complex. Parallel to the situation with Coxeter complexes, it is well known that the link of a vertex $v_0$ in a building $\Delta$ of type $\tilde{A}_{n-1}$ is a building of type $A_{n-1}$. In order to describe $A_n$, we recall the following definition. 
\begin{definition}
\cite[pg. 127]{Moo07} A finite projective geometry is a incidence geometry satisfying the following, where we call $n$ the projective dimension of the geometry: 
\begin{itemize}
    \item There is a unique line through any two points. 
\item There exist three non-colinear points. 
\item Every line contains at least three points. 
\item A chain of non-empty subspaces has length at most $n+1$. 
\item Every line that is incident with two side of a triangle, and not with the vertices of the triangle, must be incident with the third side of the triangle.
\end{itemize}

\end{definition}
A locally finite building is a building where the link of every vertex is finite. Equipped with this terminology, we can state the following lemma due to Tits.  

\begin{lemma}\cite[Theorem 6.3]{Tits74}
\label{buildingFPG}
A locally finite building of type $A_{n -1}$ is isomorphic to a finite projective geometry of projective dimension $n - 1$.
\end{lemma} 

The order of the building is defined as the order of the finite projective geometry arising as the link of (any) vertex. This gives us an interesting description of the relationship between Coxeter complexes and buildings. We recall that we can define the $q$-integer $[k] = \frac{q^{k}-q^{-k}}{q - q^{-1}}$ and $[k]!_q = [k]_q[k-1]_q...[2]_q[1]_q$ To see this, we first recall that in a finite projective geometry of algebraic dimension $n$ and order $q$, the number of subspaces of dimension $k$ is  $\bigl[\!\begin{smallmatrix} n\\ k \end{smallmatrix}\!\bigr]_q = \frac{[n]!_q}{[k]!_q[n-k]!_q}$ (see \cite[Page 121]{Moo07}).  If we set $q = 1$, this is simply equal to ${n \choose k}$. This is of course the number of subsets of size $k$ of an $n$ element set. So in the philosophy of \cite{Tits57}, a flag complex of proper non-empty subsets of $[n]$ can be thought of as a finite projective geometry of dimension $n$ over Tits' degenerate ``field of order 1". Reversing the logic, this suggests that we can think of general finite type $A_{n-1}$ buildings as $q$-analogues of the $A_{n-1}$ Coxeter complex.

Now we consider this in the affine case. Lemma \ref{buidling} and the facts stated we can say that $\tilde{A}_{n-1}$ Coxeter complexes are equivalent to $\tilde{A}_{n-1}$ buildings of ``order 1".
Since the $\tilde{A}_{n-1}$ Coxeter complex is isomorphic to the weight lattice of $sl_n$, we can say that a locally finite $\tilde{A}_{n-1}$ building is a $q$-analogue of this weight lattice. Another way to say this, is that the Coxeter complex is the degenerate building over the field of order 1 as imagined in \cite{Tits57}.

As we finish our discussion of how finite projective geometry ties into the theory of buildings, we state several facts that we will use in our proof of Lemma \ref{BuildingSS} and Theorem \ref{GPAFunctorThm}. Recall that in algebraic dimension $\geq 4$, these questions are problems in linear algebra, as we must only consider classical projective geometries. In algebraic dimension 3, we must turn to the theory of projective planes. Both of the facts below follow from the basic theory of projective planes (see \cite[Chapter 6]{Moo07} and in particular Theorem 6.3 for justification). Note that the projective dimension of the geomtries in these Lemmas is $n-1$. 

\begin{lemma}
\label{Number of K-dim subspaces}
In a finite projective geometry of algebraic dimension $n$ and order $q$, the number of subspaces of algebraic dimension $k$ is  $\bigl[\!\begin{smallmatrix} n\\ k \end{smallmatrix}\!\bigr]_q$
\end{lemma}

\begin{lemma}
\label{Number of K-dim subspaces with fixed M dim subspace}
In a finite projective geometry of algebraic dimension $n$ and order $q$, the number of subspaces of algebraic dimension $k$ containing some fixed $m$-dimensional subspace  is  $\bigl[\!\begin{smallmatrix} n-m\\ k-m \end{smallmatrix}\!\bigr]_q$
\end{lemma}
Finally, we note that for subspaces $V$ and $W$ of a vector space $U$, we have $dim(V+W) = dim(V) + dim(W) - dim(V \cap W)$. This fact also holds for projective planes, and is easy to see by inspection of the definition. 
\subsection{The Graph Planar Algebra $G(\Delta)$}
\label{GPAsection}
We can view the 1 skeleton  of a type $\tilde{A}_{n-1}$ building $ \Delta$ as a graph which we call $\Gamma_{\Delta}$, where we define the vertex and edge sets of $\Gamma_{\Delta}$ as the $0$ and 1 simplices of $\Delta$ respectively. This graph is undirected, but we can easily view it as as directed graph by replacing each edge with two directed edges with opposite sources and targets. Recall that every building has a labeling that is consistent with the labeling of its apartments. So we can label $V(\Gamma_{\Delta})$ by the elements of the set $[n]$ in a way that is consistent with  the labeling on $\Delta$ (where a vertex $v$ has label $\ell(v)$) . We use this vertex labeling to define a labeling on $E(\Gamma_{\Delta})$ by labeling an edge from $x$ to $y$ with $\ell(x) - \ell(y) \mod n$. We see then that if the edge $x \to y$ has label $k$, then edge $y \to x$ has label $n - k \text{ mod } n$. We will abuse notation slightly by using $\Gamma_\Delta$ to refer to this directed graph as well. In fact, we will almost exclusively use the directed version of $\Gamma_{\Delta}$. 

\begin{remark}
\label{link as subspaces}
\normalfont
Recall that for any vertex $v_0$ in $\Delta$ a $\tilde{A}_{n-1}$ building of order $q$, that $lk_\Delta(v_0)$ is isomorphic to a $A_{n-1}$ building. This is in turn isomorphic to a finite projective geometry of order $q$. So we can choose to identify every vertex in $lk_\Delta(v_0)$ (or equivalently all the edges in $\Gamma_{\Delta}$ originating at $v_0$) with a proper subspace of an $n$ dimensional vector space in a way that the label of an edge in $\Gamma_\Delta$ will be equal to with the dimension of the subspace to which it is identified (see the proof of \cite[Theorem 6.3]{Tits74}). Incidence in the projective geometry corresponds to inclusion of subspaces. So a cycle $v_0 \to x \to y \to v_0$ exists in $\Gamma_\Delta$ if and only if one of the subspaces associated to $x$ and $y$ is contained in the other. 
\end{remark}
Now, equipped with an encoding of any $\tilde{A}_{n-1}$ building as a directed, edge-labeled graph, we make the following definitions. 
\begin{definition} 
\label{pathtype} 
If $p$ is a path in $\Gamma_{\Delta}$ with edges $e_1, ..., e_k$, then the type of $p$ (notation $type(p)$), is the tuple of length $k$ where $(type(p))_i = \ell(e_i)$. 
\end{definition}
\begin{definition}
\label{GPA}
For a graph $\Gamma$ whose edges are labelled by natural numbers, we define the category $G(\Gamma)$ over a field $\Bbbk$ as the category whose: 
\begin{itemize}
\item Objects are finite sequences  of natural numbers (selected from the edge labels of $\Gamma$). 
\item A morphism between objects $\sigma$ and $\tau$ is a linear functional $$f: \Bbbk(\{(p_1, p_2) |\,\, type(p_1) = \sigma, \,\, type(p_2) = \tau\}) \rightarrow \Bbbk$$ where $f((p_1, p_2)) = 0$ unless  the starting and ending points of $p_1$ and $p_2$ coincide.
\item A composition of morphisms $f: \sigma \rightarrow \tau$ and $g: \tau \rightarrow \omega$ is defined on a matched pair $(p_1, p_2)$ of type $(\sigma, \omega)$ as follows:  
\begin{equation} \label{gpa composition}(g \circ f)((p_1, p_2)) = \sum_{p' \in P'} f((p_1, p')) * g((p', p_2))\end{equation} where $P'$ is the set of paths of type $\tau$ whose starting and ending vertices coincide with $p_1$ and $p_2$. 
\end{itemize}
This category is a strict monoidal category. The monoidal product acts in the following way
\begin{itemize}
\item For sequences $\sigma$ and $\tau$, $\sigma \otimes \tau = (\sigma_1, ..., \sigma_k, \tau_1, ..., \tau_{m})$. \item For morphsims $f: \sigma \to  \tau$ and $g: \omega \to  \mu$, $$(f \otimes g) ((p_1 \otimes q_1, p_2 \otimes q_2))= f((p_1, p_2)) * g((q_1, q_2))$$. 
\end{itemize}
The monoidal unit is the empty sequence. 
\end{definition}
The main focus of our discussion will be $G(\Gamma_{\Delta})$ for $\Delta$ a locally finite building of type $\tilde{A}_{n-1}$. We will use the short hand $G(\Delta)$ for this category. The categories we have described above are instances of \textit{graph planar algebras} originally introduced by V.F.R. Jones in the context of subfactors (see \cite{Jones98}, \cite{Jones99} and \cite{Mor10}) Our version is slightly different but very similar in spirit.

\begin{remark}
\label{gpa from paper}
\normalfont
Looking to the work in \cite{Mor10} we see that $G(\Delta)$ follows the authors' definition of a graph planar algebra in the following way. Let $G$ be the following graph 
$$ 
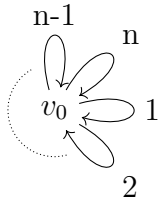
\begin{tikzpicture}[scale=1, baseline = -.1cm]
\node   (a)  {$v_0$};
\path[scale=2] 
        (a) edge [loop right, "1"]
        (a) edge [out=-30, in=-60, distance=5mm, ->, "2"] (a)
        (a) edge [loop above, "n-1"]
        (a) edge [out= 60, in= 30, distance=5mm, ->, "n"] (a);
\draw [densely dotted]  (285:1.5em) arc (285:120:1.5em);
    \end{tikzpicture}$$
 Let $\pi: \Gamma_\Delta \to G$ be the homomorphism sending every element in $V(\Gamma_{\Delta})$ to the single vertex $v_0$ in $G$ and every edge with label $i$ in $E(\Gamma_{\Delta})$ to the edge in $G$ labeled $i$. Now, we can use the structure of $\tilde{A}_{n-1}$ buildings to see that choosing $\delta_k = \bigl[\!\begin{smallmatrix} n\\ k\end{smallmatrix}\!\bigr]_q$ and $d(v) = 1$ for all $v \in V(\Delta_\Gamma)$ gives us Perron-Frobenius data for this homomorphism.
 \end{remark} 
Notice that $\Gamma_\Delta$ is not finite, and thus does not exactly satisfy \cite{Mor10}'s definition of a bidirected graph, but that this construction only requires $\Gamma_\Delta$ to be locally finite, which it is. We then see that $G(\pi)$ as defined in (\cite{Mor10}, definition 2.5) is equivalent to our previously defined $G(\Delta)$. Because of this connection, we will often refer to $G(\Delta)$ as the graph planar algebra of $\Delta$.

Finally, we introduce the following definition which will help us discuss this category locally. 

\begin{definition}
\label{composition labelling}
For a list of composable morphisms $f_1, f_2,…,  f_k$ in $G(\Delta)$ where $f_i: \sigma_i \rightarrow \sigma_{i+1}$, a labeling is a specific choice of paths of types $\sigma_1, …, \sigma_n$ with common initial and final vertices. For example each choice of $p’$ in \ref{gpa composition} gives us a labeling $(p_1, p’, p_2)$ of $f, g$.
\end{definition}

Notice that we can use this idea to reword the definition of composition in  definition \ref{GPA} by saying that we can compose a list of morphisms $f_1, f_2,…,f_k$ by summing over all labelings of this list and taking the product of values in each labeling. For a specific evaluation of this map (i.e. $f_1 \circ f_2  \circ f_ {k}(p_1, p_2)$), we will consider only labelings starting with $p_1$ and ending with $p_2$. Also, if we have a morphism that is the sum of multiple web pictures, we can evaluate it by independently summing over labelings of each summand and then adding the values together.

\subsection{The Category $Web(SL_n^{-})$}
\label{webSLn}
In this section, we introduce the category $Web(SL_n^{-})$ as an an example of a class of categories often called ``web categories." (first found in \cite{Kup96}). These categories, described by diagrammatic generators and relations, derive their name from their morphisms, as complicated compositions in these category have pictorial descriptions resembling spiderwebs. Web categories have a connection to representation theory that we will explore in subsection \ref{web context}. 
\\
\\
$Web(SL_n^{-})$ was introduced in \cite{Jones21} as an extension of the categories described in \cite{BEAEO20}. These are closely related to the $Sl_n$ web categories originally introduced in \cite{CKM14}, though our categories do not have a chosen pivotal structure. Formally, the category $Web(SL_n^{-})$ is the category whose objects are finite sequences in $[n]$ and whose morphisms are generated by the following diagrams (called webs):

$$\begin{tikzpicture}[scale=1, baseline = -.1cm]
\draw (0,0.5) to (0,0); 
\draw (-0.3, -0.5) to (0,0);
\draw (0.3, -0.5) to (0, 0);
\node(none) at (0, 0.7) {j + k};
\node(none) at (-0.3, -0.7){j};
\node(none) at (0.3, -0.7){k};
\end{tikzpicture} \text { and } \begin{tikzpicture}[scale=1, baseline = -.1cm]
\draw (0,-0.5) to (0,0); 
\draw (-0.3, 0.5) to (0,0);
\draw (0.3, 0.5) to (0, 0);
\node(none) at (0, -0.7) {j + k};
\node(none) at (-0.3, 0.7){j};
\node(none) at (0.3, 0.7){k};
\end{tikzpicture} \text{ and } \begin{tikzpicture}[scale=1, baseline = -.1cm]
\draw (0, 0.3) to (0,-0.4);
\node(none) at (0, -0.6){n};
\mydotw{(90:0.35)};
\end{tikzpicture} \text { and } \begin{tikzpicture}[scale=1, baseline = -.1cm]
\draw (0, 0.3) to (0,-0.4);
\node(none) at (0, 0.45){n};
\mydotw{(270:0.5)};
\end{tikzpicture}  $$ 
Here, $j$ and $k$ can be any natural number labels, such that $j + k \leq n$. If $j +k > n$, then the morphism does not exist. The collection of morphisms in this category is just the collection of formal $\Bbbk$ linear combinations of horizontal and vertical compositions of these generating morphisms. The morphsims satisfy the following relations.  
\begin{equation}
\label{associativity}
\begin{tikzpicture}[scale=1, baseline = -.1cm]
\draw(-1, 1) to (0,0);
\draw(1,1) to (0,0);
\draw(0,0) to (0, -1);
\draw (0, 1) to (-0.5, 0.5);
\node(none) at (0, -1.2) {$j+k+\ell$};
\node(none) at (-1, 1.2) {$j$};
\node(none) at (0, 1.2) {$k$};
\node(none) at (1, 1.2) {$\ell$};
\end{tikzpicture} = \begin{tikzpicture}[scale=1, baseline = -.1cm]
\draw(-1, 1) to (0,0);
\draw(1,1) to (0,0);
\draw(0,0) to (0, -1);
\draw (0, 1) to (0.5, 0.5);
\node(none) at (0, -1.2) {$j+k+\ell$};
\node(none) at (-1, 1.2) {$j$};
\node(none) at (0, 1.2) {$k$};
\node(none) at (1, 1.2) {$\ell$};
\end{tikzpicture} \text{ and } \begin{tikzpicture}[scale=1, baseline = -.1cm]
\draw(-1, -1) to (0,0);
\draw(1,-1) to (0,0);
\draw(0,0) to (0, 1);
\draw (0, -1) to (-0.5, -0.5);
\node(none) at (0, 1.2) {$j+k+\ell$};
\node(none) at (-1, -1.2) {$j$};
\node(none) at (0, -1.2) {$k$};
\node(none) at (1, -1.2) {$\ell$}; \end{tikzpicture} = \begin{tikzpicture}[scale=1, baseline = -.1cm]
\draw(-1, -1) to (0,0);
\draw(1,-1) to (0,0);
\draw(0,0) to (0, 1);
\draw (0, -1) to (0.5, -0.5);
\node(none) at (0, 1.2) {$j+k+\ell$};
\node(none) at (-1, -1.2) {$j$};
\node(none) at (0, -1.2) {$k$};
\node(none) at (1, -1.2) {$\ell$}; \end{tikzpicture} \end{equation}

\begin{equation} \label{bigon bursting}\begin{tikzpicture}[scale=1, baseline = -.1cm]
\draw (0,-1)node[below]{j+k} -- (0,-0.5) ;
\draw (270:0.5) arc (270:90:0.5);
\draw(-90:0.5) arc (-90:90:0.5);
\draw(0,0.5)--(0,1) node[above]{j+k};
\node(none) at (1,0) {k};
\node(none) at (-1,0) {j};
\end{tikzpicture}  = {j +k \choose j} \begin{tikzpicture}[scale=1, baseline = -.1cm]
\draw (0,-1) to (0,1);
\node(none) at (0, -1.2) {j+k};
\end{tikzpicture}\end{equation}

\begin{equation} \label{square switch 1} \begin{tikzpicture}[scale=1, baseline = -.1cm]
\draw (-0.5,-1) -- (-0.5, 1) ;
\draw (0.5, -1) -- (0.5, 1) ;
\draw (-0.5, -0.6) -- (0.5, -0.2);
\draw (-0.5, 0.6) -- (0.5, 0.2);
\node(none) at (0, 0.6) {j};
\node(none) at (0, -0.6) {k};
\node(none) at (-0.5,-1.2) {m};
\node(none) at (0.5,-1.2) {$\ell$};
\end{tikzpicture} = \sum_{t}{m-\ell+j-k \choose t}\begin{tikzpicture}[scale=1, baseline = -.1cm]
\draw (-0.5,-1) -- (-0.5, 1) ;
\draw (0.5, -1) -- (0.5, 1) ;
\draw (-0.5, -0.2) -- (0.5, -0.6);
\draw (-0.5, 0.2) -- (0.5, 0.6);
\node(none) at (0, 0.65) {k - t};
\node(none) at (0, -0.65) {j- t};
\node(none) at (-0.5,-1.2) {m};
\node(none) at (0.5,-1.2) {$\ell$};
\end{tikzpicture}\end{equation}

\begin{equation}
\label{square switch 2}
\begin{tikzpicture}[scale=1, baseline = -.1cm]
\draw (-0.5,-1) -- (-0.5, 1) ;
\draw (0.5, -1) -- (0.5, 1) ;
\draw (-0.5, -0.2) -- (0.5, -0.6);
\draw (-0.5, 0.2) -- (0.5, 0.6);
\node(none) at (0, 0.6) {j};
\node(none) at (0, -0.6) {k};

\node(none) at (-0.5,-1.2) {m};
\node(none) at (0.5,-1.2) {$\ell$};

\end{tikzpicture} = \sum_{t}{\ell - m+k -j \choose t}\begin{tikzpicture}[scale=1, baseline = -.1cm]
\draw (-0.5,-1) -- (-0.5, 1) ;
\draw (0.5, -1) -- (0.5, 1) ;
\draw (-0.5, -0.6) -- (0.5, -0.2);
\draw (-0.5, 0.6) -- (0.5, 0.2);
\node(none) at (0, 0.65) {k - t};
\node(none) at (0, -0.65) {j - t};
\node(none) at (-0.5,-1.2) {m};
\node(none) at (0.5,-1.2) {$\ell$};
\end{tikzpicture}\end{equation}

\begin{equation}
\label{lollipop relations}
\begin{tikzpicture}[scale=1, baseline = -.1cm]
\draw (0,-0.5) to (0, 0.5);
\node(none) at (0, -0.7) {n};
\end{tikzpicture} 
= \begin{tikzpicture}[scale=1, baseline = -.1cm]
\draw (0,0.5) to (0, 0.1);
\draw (0,-0.5) to (0, -0.1);
\mydotw{(270:0.1)};
\mydotw{(90:0.1)};
\node(none) at (0, -0.7) {n};
\end{tikzpicture} \text {  and  } \begin{tikzpicture}[scale=1, baseline = -.1cm]
\draw (0,-0.5) to (0, 0.5);
\mydotw{(270:0.5)};
\mydotw{(90:0.5)};
\end{tikzpicture}  = \begin{tikzpicture}[scale=1, baseline = -.1cm]
\draw[dotted] (270:0.3) arc (270:90:0.3);
\draw[dotted] (-90:0.3) arc (-90:90:0.3);
\end{tikzpicture}  \end{equation}

\begin{equation}
\label{SLn relations} \begin{tikzpicture}[scale=1, baseline = -.1cm]
\draw (0, -1) to (0, 1);
\node(none) at (0, -1.2) {m};
\end{tikzpicture} = \begin{tikzpicture}[scale=1, baseline = -.1cm]
\draw (-0.5, -1) to (-0.5, 0.8);
\draw (0.5, -0.8) to (0.5, 1);
\draw (-0.5, 0.4) to (0.5, -0.4);
\node(none) at (-0.5, -1.2) {m};
\node(none) at (0.5, 1.2) {m};
\node(none) at (0, 0.4) {n-m};
 \mydotw{(125:0.9)};
 \mydotw{(305:0.9)};
\end{tikzpicture}  = \begin{tikzpicture}[scale=1, baseline = -.1cm]
\draw (-0.5, -0.8) to (-0.5, 1);
\draw (0.5, -1) to (0.5, 0.8);
\draw (-0.5, -0.4) to (0.5, 0.4);
\node(none) at (0.5, -1.2) {m};
\node(none) at (-0.5, 1.2) {m};
\node(none) at (0, 0.4) {n-m};
 \mydotw{(55:0.9)};
 \mydotw{(235:0.9)};
\end{tikzpicture}\end{equation}
We will call relation \ref{associativity} the associativity and coassociativity relation, relation \ref{bigon bursting} the bigon bursting relation and relations \ref{square switch 1} and \ref{square switch 2} the square switch relations. We define a tensor product on objects of $Web(SL_n^{-})$ as concatenation of lists and on morphisms as horizontial composition of pictures. This construction makes $Web(SL_n^{-})$ a strict monoidal category (see \cite{EGNO15} for background information on monoidal categories). Futhermore, we see that the last relation for morphism  shows that we have duals in $Web(SL_n^{-})$. So $Web(SL_n^{-})$ is also a rigid monoidal category. 

\subsubsection{$Web(SL_n^{-})$ in Context}
\label{web context}
 Recall that a quotient of a monoidal category $C$ is a dominant monoidal functor $C \rightarrow D$ for a monoidal category $D$. One way to derive $Web(SL_n^{-})$ is as a quotient of a larger web category $PolyWeb(GL_n)$.  In fact, this is how \cite{Jones21} presents this category. 

The theory of web categories gives us a pictorial way to describe categories of representations. For example, the larger category $PolyWeb(GL_n)$ is equivalent to the category of polynomial representations of $GL_n$ over an algebraically closed field.  When $n = 2k + 1$, we have that $Web(SL_n^{-})$ is equivalent to a web category  $Web(SL_n^{+})$, with almost identical structure (the major distinction between these categories are the coefficents attached to relation 6). If we take these categories over an algebraically closed field $\Bbbk$, we have that the Karoubi envelope  of $Web(SL_n^{-})$ is equivalent to $Tilt(SL_n)$, the category of titling modules for $SL_n$ (see \cite[Remark 3.4]{Jones21}). As a result (and because the Karoubi completion is categorical \cite[Theorem 3.21]{GMPPS23},  functors from $Tilt(SL_{2k+1})$ are completely described by functors out of $Web(SL_{2k+1}^{-})$, where we must simply define images for all generating morphisms that satisfy the $Web(SL_n^{-})$ relations.

In characteristic 0, $Tilt(SL_n)$ is equivalent to the category $Rep(SL_n)$, but in positive characteristic it is not. Instead, it is simply a full subcategory of $Rep(SL_n)$. In this setting, $Tilt(SL_n)$ can be characterized as the subcategory of $Rep(SL_n)$ whose behavior mimics that of $Rep(SL_n)$ in characteristic 0. Additionally, $Rep(SL_n)$ is the Abelian envelope of $Tilt(SL_n)$ \cite{CEOP21}. 

When $n$ is even, we believe that $Web(SL_n^{-})$ corresponds to tilting modules for a $q = -1$ deformation of $SL_n$. For example, when $n = 2$, we have that $Web(SL_2^{-}) \simeq TLJ(2)$, the Temperley-Lieb-Jones category with loop parameter $\delta = 2$ (For more on this example, see \cite{Jones21}). 

\section{Embedding $Web(SL_n^{-})$ in $G(\Delta)$}
\label{functorsection}
Our goal for this section is to define a functor from $Web(SL_n^{-})$ to $G(\Delta)$ for any locally finite type $\tilde{A}_{n-1}$ building $\Delta$. We will do this by specifying the images of the generating maps in $Web(SL_n^{-}))$. We will then explore what it would mean to have these images satisfy the $Web(SL_n^{-})$ relations, and introduce the language we will use to ultimately prove the existence of this functor. 
\subsection{Defining the Maps}
\label{FunctorDef}
Recall, the generating maps of the category $Web(SL_n^{-})$ are: $$\begin{tikzpicture}[scale=1, baseline = -.1cm]
\draw (0,0.5) to (0,0); 
\draw (-0.3, -0.5) to (0,0);
\draw (0.3, -0.5) to (0, 0);
\node(none) at (0, 0.7) {j + k};
\node(none) at (-0.3, -0.7){j};
\node(none) at (0.3, -0.7){k};
\end{tikzpicture} \text { and } \begin{tikzpicture}[scale=1, baseline = -.1cm]
\draw (0,-0.5) to (0,0); 
\draw (-0.3, 0.5) to (0,0);
\draw (0.3, 0.5) to (0, 0);
\node(none) at (0, -0.7) {j + k};
\node(none) at (-0.3, 0.7){j};
\node(none) at (0.3, 0.7){k};
\end{tikzpicture} \text{ and } \begin{tikzpicture}[scale=1, baseline = -.1cm]
\draw (0, 0.3) to (0,-0.4);
\node(none) at (0, -0.6){n};
\mydotw{(90:0.35)};
\end{tikzpicture} \text { and } \begin{tikzpicture}[scale=1, baseline = -.1cm]
\draw (0, 0.3) to (0,-0.4);
\node(none) at (0, 0.45){n};
\mydotw{(270:0.5)};
\end{tikzpicture}  $$ 
Now, notice that in both $Web(SL_n^{-})$ and $G(\Delta)$ the objects are sequences of integers in $\{1,..., n\}$. We map a sequence of integers in $Web(SL_n^{-})$ to the corresponding sequence $\mod n$ in $G(\Delta)$. This allows us to identify $n \in Web(SL_n^{-})$ with 0 in $G(\Delta)$. We now  define the images of the generating maps in $G(\Delta)$
$$\text{For } j+k < n, \begin{tikzpicture}[scale=1, baseline = -.1cm]
\draw (0,0.5) to (0,0); 
\draw (-0.3, -0.5) to (0,0);
\draw (0.3,-0.5) to (0, 0);
\node(none) at (0, 0.7) {j + k};
\node(none) at (-0.3, -0.7){j};
\node(none) at (0.3, -0.7){k};
\end{tikzpicture}  \mapsto \mathbbm{1}_{j,k}: \Bbbk(\{(p_1, p_2) |\,\,\, type(p_1) = (j,k) \text{ and } type(p_2) = (j+k)\}) \rightarrow \Bbbk$$
$$\begin{tikzpicture}[scale=1, baseline = -.1cm]
\draw (0,0.5) to (0,0); 
\draw (-0.3, -0.5) to (0,0);
\draw (0.3,-0.5) to (0, 0);
\node(none) at (0, 0.7) {n};
\node(none) at (-0.3, -0.7){j};
\node(none) at (0.3, -0.7){n-j};
\end{tikzpicture} \mapsto \mathbbm{1}_{j,n-j}: \Bbbk(\{(p_1, p_2) | \,\,\, type(p_1) = (j,n-j) \text{ and } type(p_2) = \emptyset\})
$$
$$\text{For } j+k < n,\begin{tikzpicture}[scale=1, baseline = -.1cm]
\draw (0,-0.5) to (0,0); 
\draw (-0.3, 0.5) to (0,0);
\draw (0.3, 0.5) to (0, 0);
\node(none) at (0, -0.7) {j + k};
\node(none) at (-0.3, 0.7){j};
\node(none) at (0.3, 0.7){k};
\end{tikzpicture} \mapsto \mathbbm{1}_{j + k}: \Bbbk(\{(p_1, p_2) | type(p_1) = (j+k) \text{ and } type(p_2) = (j,k)\}) \rightarrow \Bbbk$$
$$\begin{tikzpicture}[scale=1, baseline = -.1cm]
\draw (0,-0.5) to (0,0); 
\draw (-0.3, 0.5) to (0,0);
\draw (0.3, 0.5) to (0, 0);
\node(none) at (0, -0.7) {n};
\node(none) at (-0.3, 0.7){j};
\node(none) at (0.3, 0.7){n-j}; 
\end{tikzpicture} \mapsto \mathbbm{1}_{\emptyset}: \Bbbk(\{(p_1, p_2) | \,\,\, type(p_1) =\emptyset \text{ and } type(p_2) =  (j,n-j)\})$$
$$
\begin{tikzpicture}[scale=1, baseline = -.1cm]
\draw (0, 0.3) to (0,-0.4);
\node(none) at (0, -0.5){n};
\mydotw{(90:0.35)};
\end{tikzpicture} \mapsto \mathbbm{1}_{\emptyset, \emptyset}: \Bbbk(\{(p_1, p_2 | \,\, type(p_1) = (\emptyset) \text{ and } type(p_2) = \emptyset\}) = \Bbbk \rightarrow \Bbbk$$
$$\begin{tikzpicture}[scale=1, baseline = -.1cm]
\draw (0, 0) to (0,-0.7);
\node(none) at (0, 0.2){n};
\mydotw{(270:0.7)};
\end{tikzpicture} \mapsto \mathbbm{1}_{\emptyset, \emptyset}: \Bbbk(\{(p_1, p_2 | \,\, type(p_1) = \emptyset \text{ and } type(p_2) = (\emptyset)\}) = \Bbbk \rightarrow \Bbbk$$
Here, $\emptyset$ represents the path with no edges. We must use this when we have a label $n$, since we do not have an edge label $n \equiv 0$ in $\Gamma_\Delta$ (this comes from the fact that the two distinct vertices of the same type cannot be connected in the $\tilde{A}_{n-1}$ Coxeter complex). The map $\mathbbm{1}_{j, k}$ evaluates to 1 on a matched pair $(p,q)$, if $pq$, the concatenation of $p$ and $q$ is a cycle of length 2 or 3, and 0 otherwise (for $\mathbbm{1}_{\emptyset, \emptyset}$, this simply evaluates to 1 on every vertex). Similarly, the other maps defined also evaluate to $1$ on every triangle and 0 elsewhere. To show that this construction indeed gives us a functor $Web(SL_n^{-}) \rightarrow G(\Delta)$, we must show that the maps defined above satisfy the $Web(SL_n^{-})$ relations.

Since the $Web(SL_n^{-})$ relations are pictorial, in order to show that the maps we have chosen in $G(\Delta)$ satisfy these relations we wish to have a pictorial understanding of them as well. To do this, we notice that the generator of $Web(SL_n^{-})$ correspond to triangles in $\Gamma_\Delta$. Specifically, each string in one of these generator represents an edge whose label matches that of the string. As a result, each region bounded by two or more strings can be labeled with a vertex of $\Gamma_\Delta$. In order for a composition to have a valid labeling as defined in \ref{composition labelling} we must have a collection of compatible triangles in $\Gamma_\Delta$ (by compatible, we mean triangles that share an edge where two $Web(SL_n^{-})$ generators are composed). Also, note that the maps $\mathbbm{1}_{j,k},\mathbbm{1}_{j+k},\mathbbm{1}_{j,n-j}, \mathbbm{1}_{\empty}$ take the values 0 or 1. Any composition of these maps is a function that simply counts the number of triangle arrangements that satisfy this composition. Of course, as before, evaluating a map at a specific pair of paths fixes the initial and final paths in a labeling. See that this fixes some vertices in our triangle arrangement as well (If we draw this arrangement of triangles on top of our pictorial composition, these fixed vertices will appear in the boundary regions of our picture).

\subsection{The Degenerate Case}
\label{degernate case section}
To aid in our proof of the existence of the desired functor for all $\tilde{A}_{n-1}$ buildings, we explore the special case where $q = 1$. In this case, note that as previously discussed we can treat a collection of finite subsets of a set of size $n$ as a finite projective geometry of order 1. We will now show in detail, the connection between the $sl_n$ weight lattice and the $\tilde{A}_{n-1}$ Coxeter complex. We make the following definition to provide an alternative presentation of the weight lattice. This definition will help us to prove the existence of the functor we defined in the previous section. 
\begin{definition}
\label{degenerateTP}

 \cite[Example 2.11]{Jones21} For some natural number $n$, let $x = \{1,...,n\}$, $S$ be the collection of proper nonempty subsets of $X$ and define a function $\sigma: S \rightarrow S$ where $\sigma(x) = X \setminus x$ for all $x \subset X$. We can partition $X$ into subsets $\Pi_i$  for $1 \leq i \leq n-1$, where $\Pi_m$ is the set of subsets of size $m$.  Then if we set $$ T_1 = \{(A, B,C) \in S \times S \times S | A,B,C \text{ are pairwise disjoint, and } A \cup B \cup C = X\}$$ $$ T_2 = \{(A,B,C) \in S \times S \times S| (\sigma(C), \sigma(B), \sigma(A)) \in T_1\}$$
 then we can define $\cT = T_1 \cup T_2$. We can also define a group $$\Gamma_\cT = \langle g_a| a \in S, g_ag_bg_c = 1 \text{ if and only if } (a,b,c) \in \cT, g_ag_{\sigma(a)} = 1\rangle.$$
\end{definition}

 Note that this means  $A, B, C \in T_1$ if and only if $A \cap B = \emptyset$ and $\sigma(C) = A \cup B$. We will occasionally call $\cT$ the degenerate triangle presentation, from the work of \cite{CMSZ93}, as the $q$-analogue of the structure we have defined is Cartwright's $\tilde{A}_{n-1}$ triangle presentations. In particular $\cT$ satisfies the following properties. 
\begin{itemize}
  \item \cite[Proposition 2.5]{Jones21} If $dim(A)+dim(B) + dim(C) < n$, then $(A,B, \sigma(D)), (D, C, \sigma(E)) \in \cT$ if and only if $(A, F, \sigma(E)), (B, C, \sigma(F)) \in \cT$, where $D, E, F$ are unique.
\end{itemize}

\begin{lemma}
\label{T_1 condition}
If $(A,B,C) \in \cT$, then $|A| + |B| < n$ if and only if $(A,B,C) \in T_1$. 
\end{lemma}
\begin{proof}
    First, we show that if $(A,B,C) \in T_1$, then $|A| + |B| < n$. To see this, see $(A,B,C) \in T_1$ means that $A, B$ and $C$ are disjoint sets whose union is $X$, then $|A| + |B| + |C| = n$, so $|A| + |B| = n - |C| < n$. 
    \\
    \\
    Now, to show the converse, we will show that if $(A', B', C') \in T_2$, then $|A'| + |B'| > n$. To do this, see that $(A', B', C') \in T_2$ implies $(\sigma(C'), \sigma(B'), \sigma(A')) \in T_1$. By the rotational symmetry of triangle presentations, we have $(\sigma(B'), \sigma(A'), \sigma(C')) \in T_1$ as well. See that now from the first paragraph, $|\sigma(B')| + |\sigma(A')| < n$. But now $|A'| + |B
'|  = n - |\sigma(A')| + n - |\sigma(B')| = 2n - ( |\sigma(B')| + |\sigma(A')|) > n$. So $|A'| + |B'| > n$ for every element of $T_2$, which completes the proof. 
\end{proof}
\begin{prop}
$\Gamma_{\cT}$ is isomorphic to the $sl_n$ weight lattice $\Lambda_{W}$. 
\end{prop}
\begin{proof}
To construct an isomorphism $\phi: \Lambda_W \to \Gamma_{\cT}$, we first come up with an alternate notation for an element $g$ of $\Gamma_\cT$. Since $g = g_{A_1} ... g_{A_k}$ for subsets $A_1, ..., A_k \subseteq X$, we can represent $g$ as the multiset that is the union of $A_1,..., A_k$. For example $g_{\{1\} }g_{\{1,2\}}$ is identified with the multiset $\{1,1,2\}$ or $\{1^{2}, 2^{1}\}$. This description is well defined since all generators in $\Gamma_\cT$ commute. Now we can define $\phi$ as the map that takes the weight $a_1L_1 + ... + a_nL_n$ to the element of $\cT$ represented by $\{1^{a_1}, ..., n^{a_n}\}$ (where $L_i$ are as described in \cite[pg 414]{FH04}. It is easy to see that $\phi$ is a homomorphism.  
\\
\\
We must now show that if two weights are connected in the weight lattice, that the corresponding vertices are connected in $\cT$. We consider the weights $\alpha = a_1L_1 + ... + a_nL_n$ and $\beta = b_1L_1 + ... + b_nL_n$. This means that $(a_1L_1 + ... + a_nL_n) - (b_1L_1 + ... + b_nL_n) = c_1L_1 + ... + c_nL_n$ where $c_i \in \{0, 1\}$ (this is since all weights directly adjacent to 0, and therefore appear as the difference between adjacent weights, are $\Sigma_{i = 1}^{n} a_i L_i$, with $a_i \in \{0,1\}$. So $\phi(\alpha)$ and $\phi(\beta)$ will differ by a multiset where each element has multiplicity $0$ or $1$, which is just a subset of $\{1,...,n\}$. In $\Gamma_{\cT}$, the multisets of $\phi(\alpha)$ and $\phi(\beta)$ differing by a subsets of $\{1, .., n\}$ means that the group elements differ by a generator, and are connected in $\Gamma_\cT$. 
\\
\\
To show that $\phi$ is a isomorphism, first note that surjectivity follows from  $a_1L_1 + ... + a_nL_n$ being in $\Lambda_W$ for all sequences $a_1, ..., a_n$. The only question for injectivity is that  weights can be described by multiple sequences of integers. This is due to the fact that $L_1 + ... + L_n = 0$ in the weight lattice. Suppose we have a weight $L_{i_1} + ... + L_{i_n} = - L_{j_1} + ... + - L_{j_n}$. See that $\phi(L_{i_1} + ... + L_{i_n}) = g_{\{i_1, ..., i_n\}}$ and $\phi(- L_{j_1} + ... + - L_{j_n}) = g^{-1}_{\{j_1, ..., j_n\}}$. The relation on the weight lattice means that $\{i_1, ..., i_n\} \cup \{j_1, ..., j_n\} = \{1, ..., n\}$ where the union is disjoint, and therefore .$ g_{\{i_1, ..., i_n\}} = g^{-1}_{\{j_1, ..., j_n\}}$ by definition of $\Gamma_\cT$. This argument extends to any two sequences of integers describing the same weights because $\phi $ is a homomorphism, and so $\phi$ is injective.  
\end{proof}

Recall that the facts that we stated in section \ref{CCs and Buildings} give us a connection between edges in a Coxeter complex $\Sigma$ and subsets of a set of order $n$. If we take the graph planar algebra of $\Sigma$, we will see that the order of the subsets corresponds directly to the label of the edge. This definition motivates the following theorem. 
\begin{theorem} 
\label{CayleyGraphCCISO}
If $\cT$ is the degenerate triangle presentation of type $\tilde{A}_{n-1}$, then the Cayley Graph of the group 
$\Gamma_\cT$
has flag complex isomorphic to $\Sigma$, the Coxeter complex of type $\tilde{A}_{n-1}$ (or equivalently, the Cayley graph of $\Gamma_\cT$ (with generators corresponding to subsets of $X$) is isomorphic to the 1-skeleton of $\Sigma$).
\end{theorem}

This theorem is a reimagining of \cite[Theorem 2.5]{Cart95} for the $\mathbb{F}_1$ case, and the proof follows similarly to the proof of that theorem. It holds since the Coxeter complex is the building of order 1, and $\cT$ as defined is the triangle presentation of order 1. 

An important consequence of theorem \ref{CayleyGraphCCISO} is that all triangles in the $\tilde{A}_{n-1}$ Coxeter complex must have edge labels summing to $n$ (this follows from property E on page 46 on \cite{Cart95}). This must hold for any $\tilde{A}_{n-1}$ building as well, since the edge labeling is derived from the vertex labeling, which is consistent between a building and any of its apartments. 

The connection between the Cayley graph of $\Gamma_\cT$ and the 1-skeleton of $\Sigma$ give us a correspondence between cycles $a \to b \to c \to a$ in the graph planar algebra of $\Sigma$ and elements $(u,v,w) \in \cT$ (namely, that $(u,v,w)$ are the subsets labeling the edges in the cycle). So we can say that the collection of cycles in $G(\Sigma)$ also satisfies the following property.
\begin{property}
\label{tetrahedron for CC}
If there exists vertices $a,b,c,d \in \Sigma$ so that labels of edges $a \to d$, $d \to b$ and $b \to c$ sum to less than $n$, $a \to b \to c \to a$ and $a \to d \to b \to a$ are cycles in $G(\Sigma)$ if and only if $a \to c \to d \to a$ and $b \to d \to c \to b$ are cycles in $G(\Sigma)$. That is, we have the left picture as a subgraph of $G(\Sigma)$ if any only if we have the right picture as a subgraph of $G(\Sigma)$
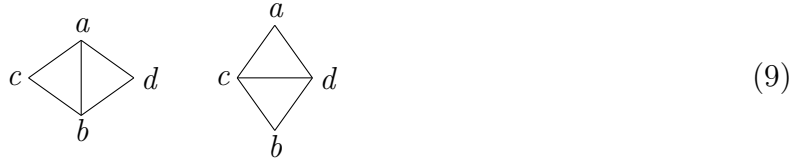
\begin{equation}
\label{tetrahedron picture}
\begin{tikzpicture}[scale=1, baseline = -.1cm]
\draw (0, 0.5) to (0, -0.5);
\draw (0, 0.5) to (-0.7, 0);
\draw (0, -0.5) to (-0.7, 0);
\draw (0, 0.5) to (0.7, 0);
\draw (0, -0.5) to (0.7, 0);
\node(none) at (0, 0.7) {a};
\node(none) at (0, -0.7) {b};
\node(none) at (0.9, 0) {d};
\node(none) at (-0.9, 0) {c};
\end{tikzpicture}
\,\,\,\,\,\,\,
\begin{tikzpicture}[scale=1, baseline = -.1cm]
\draw (0.5,0) to ( -0.5,0);
\draw (0.5,0) to ( 0, -0.7);
\draw (-0.5,0) to ( 0, -0.7);
\draw (-0.5,0) to ( 0, 0.7);
\draw (0.5,0) to ( 0, 0.7);
\node(none) at (0, 0.9) {a};
\node(none) at (0, -0.9) {b};
\node(none) at (0.7, 0) {d};
\node(none) at (-0.7, 0) {c};
\end{tikzpicture}
\end{equation}
\end{property}

\cite{Jones21} gives a definition of a functor from $Web(SL_n^{-})$ to $Vec(\Gamma_\cT)$ where $\cT$ is any triangle presentation, for defining field of characteristic $p \geq n -1$ and $q\equiv 1 \mod p$. We will show that when $\cT$ is the degenerate triangle presentation as defined in  definition \ref{degenerateTP} (and therefore is the $sl_n$ weight lattice), the restriction on the characteristic of the field is unnecessary. We show this in the following theorem.

\begin{theorem}
\label{degenSS}
If $\Gamma$ is the $sl_n$ weight lattice, then there exists a functor $Web(SL_n^{-}) \rightarrow G(\Gamma)$ as defined in section \ref{FunctorDef} and \cite[Section 4]{Jones21}, where the relevant categories may be defined over any field $\Bbbk$.
\end{theorem}
\begin{proof}
Recall that we must show that the images of the $Web(SL_n^{-})$ maps in $Vec(\Gamma)$ (as defined in section \ref{FunctorDef} satisfy the $Web(SL_n^{-})$ relations. We will show this by using the presentation of the weight lattice stated in definition \ref{degenerateTP}. This allows us to assign to each edge of the weight lattice (and therefore, strings in the $Web(SL_n^{-})$ maps and relations) a proper subset of $[n]$. Our first goal is to prove the square switch relations in full generality (\cite[Lemma 4.1]{Jones21} gives a proof where $\text{char}(\Bbbk) = p \geq n - 1$). We will show the proof of (\ref{square switch 1}). First, we will denote the functions in $Hom(m \otimes {\ell}, {m - (k-j)} \otimes {\ell + (k-j)})$ given by the left and right hand side of the equation as L and R respectively. So for some $(z ,u) \in \Pi_m \times \Pi_\ell$, we have that 
$$L(z \otimes u) = \sum_{(w,v) \in \Pi_{m -(k-j)} \times \Pi_{\ell + (k-j)} }L_{z,u,w,v} w \otimes v$$ $$R(z \otimes u) = \sum_{(w,v) \in \Pi_{m -(k-j)} \times \Pi_{\ell + (k-j)} }R_{z,u,w,v} w \otimes v.$$
We will show that $L_{z,u,w,v} = R_{z,u,w,v}$ for all $z,u,w,v$. To see this, first note that by definitioni of $Web(SL_n^{-})$, the elements of the triangle presentation represented by the generating morphisms will always have $|A| + |B| < n$ and so we can say $(A,B, \sigma(C)) \in T_1$ by Lemma \ref{T_1 condition}, which implies that $A \cup B = C$. We will make heavy use of this fact in our argument. Also, we will use the shorthand $\delta = k-j$ throughout. 

First, fix $z \in \Pi_m, w \in \Pi_\ell, w \in \Pi_{m- \delta}, v \in \Pi_{\ell+ \delta}$, with $|z \cap w| = m -i$ for some $k$. If $j \leq k$, we have $i\in [\delta, d]$ and if $c >d$ we have $i \in [0, d]$. Note that $L_{z,u,w,v}$ is the number of tuples $(p,r,q,s)$ that satisfy the left hand side of the picture below (i.e. choices of subset so that every trivalent vertex is labeled by an element of $\cT$). Similarly, $R_{z,u,w,v}$ is sum of all allowable $t$ of the number of tuples $(p',r',q',s')$ that satisfy the right hand side of the picture below. 

$$\begin{tikzpicture}[scale=1, baseline = -.1cm]
\draw (-0.5,-1) -- (-0.5, 1) ;
\draw (0.5, -1) -- (0.5, 1) ;
\draw (-0.5, -0.6) -- (0.5, -0.2);
\draw (-0.5, 0.6) -- (0.5, 0.2);
\node(none) at (0, 0.6) {s};
\node(none) at (0, -0.6) {r};
\node(none) at ( 0.8,0) {q};
\node(none) at ( -0.8,0) {p};
\node(none) at (-0.5,-1.2) {z};
\node(none) at (0.5,-1.2) {u};
\node(none) at (-0.5,1.2) {w};
\node(none) at (0.5,1.2) {v};
\end{tikzpicture} = \sum_{t}{m-\ell+j-k \choose t}\begin{tikzpicture}[scale=1, baseline = -.1cm]
\draw (-0.5,-1) -- (-0.5, 1) ;
\draw (0.5, -1) -- (0.5, 1) ;
\draw (-0.5, -0.2) -- (0.5, -0.6);
\draw (-0.5, 0.2) -- (0.5, 0.6);
\node(none) at (0, 0.6) {s'};
\node(none) at (0, -0.6) {r'};
\node(none) at ( 0.8,0) {q'};
\node(none) at ( -0.8,0) {p'};
\node(none) at (-0.5,-1.2) {z};
\node(none) at (0.5,-1.2) {u};
\node(none) at (-0.5,1.2) {w};
\node(none) at (0.5,1.2) {v};
\end{tikzpicture}$$

First, suppose that we have a valid tuple $(p,q,r,s)$, where $(p,r,\sigma(z)), (r,u,\sigma(q)), (s,v, \sigma(q)),$ $ (p,s,\sigma(w)) \in \cT$. See that if $|z \cap w| = m-i$, then $|z \setminus w| = i$ and $|w \setminus z| = i - \delta$. Since $p \cup r = z$ and $p \cup s = w$, then we have $p \subseteq z \cap w$. We must then have $z\setminus w \subseteq r$ and $w \setminus z \subseteq s$. Also, we must have $z \setminus w \not\subseteq s $ and $w \setminus z \not\subseteq r $, $r \cap u = \emptyset$ and $s \cap v = \emptyset$. So we must have $z \setminus w \subseteq v \setminus u$ and $w \setminus z \subseteq u\setminus v$. If these two conditions are not met, $L_{z,u,w,v} = 0$. So we assume that they are true and proceed with our proof. 

Now, we can say that $p = (z \cap w) \setminus j$, where $|j| = k-i$. See that this forces $r = (z\setminus w) \cup j$, $s = (w\setminus z) \cup j$ and $q = u \cup j \cup (z \setminus w)$. These values will make the $(p,r,q,s)$ satisfy this morphism. We are free to pick any subset $j \subset (z \cap w) \setminus u$. So we have $L_{z,u,w,v} = { |(z \cap w) \setminus u| \choose k-i}$. 

Now, consider the right hand side diagram and suppose we have a valid tuple $(p',r', q', s')$. See that we have $z \cup r' = p'$ and $w \cup s' = p'$ and so we must have that $z \cup w \subseteq p'$. Also, see that we must then have $w \setminus z \subseteq r'$ and $z \setminus w \subseteq s'$. Also, see we have $z \setminus w \not\subseteq r'$, $w \setminus z \not\subseteq s'$ and so we have $z \setminus w \subseteq v \setminus u$ and $w \setminus z \subseteq u\setminus v$. If these two conditions are not met, $R_{z,u,w,v} = 0$. So we have $L_{z,u,w,v}$ and $R_{z,u,w,v}$ are 0 under the same conditions. 

Now, see that $p' = (z \cup w) \cup j'$ where $|j'| = k-t-i$. Also, see that this forces $s' = (z \setminus w) \cup j'$, $r' = (w\setminus z) \cup j'$ and $q' = u \setminus r'$. We are free to pick any subset $j \subseteq u \setminus (z \cup w)$. Also, we must have $k-i-t \geq 0$, so we will have $t \in [0, k-i]$. So $R_{z,u,w,v} = \sum_{t = 0}^{k-i} {m-\ell-\delta \choose t} {|u\setminus (z \cup w)| \choose k-i-t}$. 

We must show that 
$${ |(z \cap w) \setminus u| \choose k-i} = \sum_{t = 0}^{k-i} {m-\ell-\delta \choose t} {|u\setminus (z \cup w)| \choose k-i-t}$$
We will do this by showing $|(z \cap w) \setminus u| = m - \ell -\delta + |u\setminus (z \cup w)|$ and then applying the Chu-Vandermonde identity. 

First, see that $|z \cap w \setminus u| = |z \cap w| - |u \cap z \cap w|$ and similarly $|u \setminus (z \cup w)| = |u| - |u \cap (z \cup w)|$. We can further decompose $|u \cap (z \cup w)|$ as $|u \cap (z\setminus w)| + |u \cap (w \setminus z)| + |u \cap z \cap w|$. Substituting these into desired equality, we now have
$$|z \cap w| - |u \cap z \cap w| = m -\ell - \delta + |u| - |u \cap (z\setminus w)| - |u \cap (w \setminus z)| - |u \cap z \cap w|$$ Now, we know that $z \setminus w \subset v \setminus u$, so the third to last term is simply 0. Also, $|u \cap (w \setminus z)| = |w \setminus z|$ since $w\setminus z \subset u \setminus v$. So we can further simplify to $$|z \cap w| = m - \ell - \delta + |u| - |w\setminus z|$$ Remembering $|z \cap w| = m - i$, $|u| = \ell$ and $|w \setminus z| = i - \delta$ will give us the desired equality. Thus (5) holds. The proof of (6) follows similarly. 

The proof of the other $Web(SL_n^{-})$ relations follows exactly like the proof of theorem 4.2 in \cite{Jones21}.
\end{proof}
\subsection{The Functor for Buildings}
\label{functor for buidlngs}
Now, we move towards a functor from $Web(SL_n^{-})$ to the graph planar algebra of any building $\Delta$ of type $\tilde{A}_{n-1}$. We will lean heavily on the correspondence between edges connected to a distinguished vertex in $\Delta$ and elements of a finite projective geometry of order $n$ that we established in section \ref{buidling}. When we examine the graph planar algebra of $\Delta$, the labels of edges in our graph will coincide with the dimensions of the corresponding subspaces. In the degenerate case studied above, we were able to extend our labeling of edges at one vertex to a consistent labeling of every edge in the diagram. However, in general, we do not have the symmetry needed for this labeling. Instead, we must center any argument on a specific vertex and only use subspaces to describe the link of this vertex. 

We must consider one more property of an $\tilde{A}_{n-1}$ building. This property will prove essential to showing that the square switch relations are satisfied in the graph planar algebra of $\Delta$. 

\begin{lemma} [The Tetrahedron Property]
\label{TetraProp} Suppose $\Delta$ is a locally finite building of type $\tilde{A}_{n-1}$ and $\Gamma_{\Delta}$ is as defined in section \ref{GPAsection}. If cycles $a \to b \to c \to a$ and $a \to d \to b \to a$ exist in $\Gamma_{\Delta}$, where the sum of labels on edges $a \to d,\,\, d \to b $ and $b \to c$ is less than $n$, then $c$ and $d$ are connected and therefore the cycles $a \to c \to d \to a$ and $b \to d \to c \to b$ exist. 
\end{lemma}
\begin{proof}
Consider the 2-simplices $C$ and $D$ with vertex sets $a,b,c$ and $a,b,d$ respectively. By the axioms of a building, these simplices must lie in a common apartment. Hence there exists some Coxeter complex $\Sigma_{CD} \in \mathcal{A}$ so that $a,b,c,d$ are vertices in $\Sigma_{CD}$. So in the graph $\Gamma_{\Sigma_{CD}}$, we have cycles $a \to b \to c \to a$ and $a \to d \to b \to a$. Now, use property \ref{tetrahedron for CC} to show that $a \to c \to d \to a$ and $b \to d \to c \to b$ exists in $\Gamma_{\Sigma_{CD}}$. But since $\Gamma_{\Sigma_{CD}}$ is embedded in $\Gamma_{\Sigma}$, we have that these cycles exist in $\Gamma_{\Sigma}$ as well. 
\end{proof}
This property derives its names from it's implication that if we have two faces of a tetrahedron in our building $\Delta$ we must have the entire tetrahedron in $\Delta$ as well. Equipped with this useful property and the language of finite projective geometry to describe our buildings locally, we now prove the special case of the square switch relations in the graph planar algebra of $\Delta$ shown in \cite{Jones21}. 
\begin{lemma} 
\label{BuildingSS}
For a locally finite $\tilde{A}_{n-1}$ building $\Delta$ of order $q$, if $\Bbbk$ is a field of characteristic $p$, where $q \equiv 1 \mod p$, then the following special cases of the square switch relation are satisfied by our embedding of $Web(SL_n^{-})$ into the graph planar algebra of $\Delta$ over $\Bbbk$. 
\begin{equation}
\label{square switch special case}
\begin{tikzpicture}[scale=1, baseline = -.1cm]
\draw (-0.5,-1) -- (-0.5, 1) ;
\draw (0.5, -1) -- (0.5, 1) ;
\draw (-0.5, -0.6) -- (0.5, -0.2);
\draw (-0.5, 0.6) -- (0.5, 0.2);
\node(none) at (0, 0.6) {1};
\node(none) at (0, -0.6) {1};
\node(none) at (-0.9,0) {m-1};
\node(none) at (0.9,0) {2};
\node(none) at (-0.5,-1.2) {m};
\node(none) at (0.5,-1.2) {1};
\node(none) at (-0.5,1.2) {m};
\node(none) at (0.5,1.2) {1};
\end{tikzpicture} = 
\begin{tikzpicture}[scale=1, baseline = -.1cm]
\draw (-0.5,-1) -- (-0.5, 1) ;
\draw (0.5, -1) -- (0.5, -0.6) ;
\draw (0.5, 0.6) -- (0.5, 1) ;
\draw (-0.5, -0.2) -- (0.5, -0.6);
\draw (-0.5, 0.2) -- (0.5, 0.6);
\node(none) at (-1,0) {m+1};
\node(none) at (-0.5,-1.2) {m};
\node(none) at (0.5,-1.2) {1};
\node(none) at (-0.5,1.2) {m};
\node(none) at (0.5,1.2) {1};
\node(none) at (0, 0.6) {1};
\node(none) at (0, -0.6) {1};
\end{tikzpicture} + (m-1) 
\begin{tikzpicture}[scale = 1, baseline = -.1cm]
\draw (-0.5,-1) -- (-0.5, 1.2) ;
\draw (0.5, -1) -- (0.5, 1.2) ;
\node(none) at (-0.5,-1.2) {m};
\node(none) at (0.5,-1.2) {1};
\end{tikzpicture}
\end{equation}
and 
\begin{equation}
\label{square switch special case 2}\begin{tikzpicture}[scale=1, baseline = -.1cm]
\draw (-0.5,-1) -- (-0.5, 1) ;
\draw (0.5, -1) -- (0.5, 1) ;
\draw (-0.5, -0.2) -- (0.5, -0.6);
\draw (-0.5, 0.2) -- (0.5, 0.6);
\node(none) at (0, 0.6) {1};
\node(none) at (0, -0.6) {1};
\node(none) at (-0.9,0) {2};
\node(none) at (0.9,0) {m-1};
\node(none) at (-0.5,-1.2) {1};
\node(none) at (0.5,-1.2) {m};
\node(none) at (-0.5,1.2) {1};
\node(none) at (0.5,1.2) {m};
\end{tikzpicture} = 
\begin{tikzpicture}[scale=1, baseline = -.1cm]
\draw (0.5,-1) -- (0.5, 1) ;
\draw (-0.5, -1) -- (-0.5, -0.6) ;
\draw (-0.5, 0.6) -- (-0.5, 1) ;
\draw (-0.5, -0.6) -- (0.5, -0.2);
\draw (-0.5, 0.6) -- (0.5, 0.2);
\node(none) at (0,0) {m+1};

\node(none) at (-0.5,-1.2) {1};
\node(none) at (0.5,-1.2) {m};
\node(none) at (-0.5,1.2) {1};
\node(none) at (0.5,1.2) {m};
\node(none) at (0, 0.6) {1};
\node(none) at (0, -0.6) {1};
\end{tikzpicture} + (m-1) 
\begin{tikzpicture}[scale = 1, baseline = -.1cm]
\draw (-0.5,-1) -- (-0.5, 1.2) ;
\draw (0.5, -1) -- (0.5, 1.2) ;
\node(none) at (-0.5,-1.2) {1};
\node(none) at (0.5,-1.2) {m};
\end{tikzpicture}
\end{equation}

\end{lemma}

\begin{proof}
We will give the proof of the first relation, with the second being analogous. We must show that the left and right hand sides of \ref{square switch special case} have the same evaluation for every pair of paths $(p_1, p_2)$ where $p_1, p_2$ have type $(m, 1)$. The choice of $(p_1, p_2)$ will fix the vertices labeling the boundary regions of our diagram. As mentioned in section \ref{GPAsection}, evaluating each side of the equation amounts to summing over all labelings of each side whose first element is $p_1$ and final element is $p_2$. Note that since the image of the generating morphisms in $Web(SL_n^{-})$ evaluate to either 0 or 1 for any input paths, we need not worry about multiplying values in our sum, instead just adding 1 to our total for every valid labeling. Recall also from section \ref{GPAsection} that on the right hand side we will evalute each morphism independently and add the result. To begin, we set $p_1 = a \to b \to c$ and $p_2 = a \to d \to c$, and note that this totally determines the right hand side, while on the left side we will count over all $e$ such that $p' = a \to e \to c$ gives us a labelling $(p_1, p', p_2)$. This will correspond to the following diagram. 

$$\begin{tikzpicture}[scale=1, baseline = -.1cm]
\draw (-0.5,-1) -- (-0.5, 1) ;
\draw (0.5, -1) -- (0.5, 1) ;
\draw (-0.5, -0.6) -- (0.5, -0.2);
\draw (-0.5, 0.6) -- (0.5, 0.2);
\node(none) at (-0.85,0) {a};
\node(none) at (0.85,0) {c};
\node(none) at (0, -0.8){b};
\node(none) at (0, 0.8){d};
\node(none) at (0,0){e};
\end{tikzpicture} = 
\begin{tikzpicture}[scale=1, baseline = -.1cm]
\draw (-0.5,-1) -- (-0.5, 1) ;
\draw (0.5, -1) -- (0.5, -0.6) ;
\draw (0.5, 0.6) -- (0.5, 1) ;
\draw (-0.5, -0.2) -- (0.5, -0.6);
\draw (-0.5, 0.2) -- (0.5, 0.6);
\node(none) at (-0.85,0) {a};
\node(none) at (0.35,0) {c};
\node(none) at (0, -0.8){b};
\node(none) at (0, 0.8){d};
\end{tikzpicture} + (m-1) *
\begin{tikzpicture}[scale = 1, baseline = -.1cm]
\draw (-0.5,-1) -- (-0.5, 1) ;
\draw (0.5, -1) -- (0.5, 1) ;
\node(none) at (-0.85,0) {a};
\node(none) at (0.85,0) {c};
\node(none) at (0, -0.8){b};
\node(none) at (0, 0.8){d};
\end{tikzpicture}
$$

 First, note that we must have $m < n$, since we cannot have edges with label $n$ in $\Gamma_\Delta$. If $m =1$, this relation is trivial. So we assume $1 < m < n$.
We define $L_{a,b,c,d}$ as the number of labelings of the left hand side consistent with our choice of boundary vertices, and $R_{a,b,c,d}$ as the number of labelings of the right hand side consistent with our chosen boundary. We must show that $L_{a,b,c,d} = R_{a,b,c,d}$ when evaluated in the field $\Bbbk$.

First, let's assume that $b \neq d$. We will show that $L_{a,b,c,d} \leq 1$ and $R_{a,b,c,d} \leq 1$ and then that $L_{a,b,c,d} = 1$ precisely when $R_{a,b,c,d} = 1$. Suppose $L_{a,b,c,d} \neq 0$. We then have an $e$ such that we have edges $ a \to e$ of type $m - 1$, $e \to b$ and $e \to d$ of type $1$, and $e \to c$ of type 2 in $\Gamma_\Delta$. Then we have cycles $a \to e \to b \to a$ and $a \to e \to d \to a$ in $\Gamma_\Delta$. Now, since $e$ and $b$ are connected to both $a$ and each other in $\Delta$ they are also connected in $lk_{\Delta}(a)$. We can apply remark \ref{link as subspaces} to $lk_{\Delta}(a)$ and assign a subspace $V(a \to x)$ for each edge from $a$ to $x \in lk_{\Delta}(a)$. Recall that the dimension of $V(a \to x)$ corresponds to the label of $a \to x$ (i.e. $V(a \to b)$ and $V(a \to d)$ are $m$-dimensional, while $V(a \to e)$ is $m-1$ dimensional). So by remark \ref{link as subspaces}, we have $V(a \to e) \subset V(a \to b)$ and similarly $V(a \to e) \subset V(a \to d)$ (which implies $V(a \to e) \subset V(a \to b) \cap V(a \to d)$). Now, $b \neq d$ implies $V(a \to b) \neq V(a \to d)$. Since $dim(V(a \to b)) = dim(V(a \to d)) = m$, then $V(a \to b) \cap V(a \to d)$ has dimension at most $m - 1$. $V(a \to e)$ is $m-1$ dimensional and so we must have that $V(a \to e) = V(a \to b) \cap V(a \to d)$. So if $dim ( V(a \to b) \cap V(a \to d)) < m-1$, $L_{a,b,c,d} = 0$. If $dim( V(a \to b) \cap V(a \to d)) = m-1$, then $V(a \to b) \cap V(a \to d)$ labels a unique vertex, and this vertex is the only possible $e$. So $L_{a,b,c,d}$ is at most 1. 

Now we show that $R_{a,b,c,d}$ is at most 1. First, notice that $b \neq d$ implies that the second term of this sum is 0. Now, notice that all regions of the first term are predetermined by our choice of $a,b,c,d$. So there is at most one way to fill in the diagram. We now give a condition for the existence of this filling. If we associate edges $a \to x$ with vector spaces as before, we see that $V(a \to b) \subset V(a \to c)$ and $V(a \to d) \subset V(a \to c)$. So we have that $V(a \to b) + V(a \to d) \subset V(a \to c)$. Edge $a \to c$ is type $m + 1$ and therefore $V(a \to c)$ has dimension $m + 1$. However, $b \neq d$, implies $dim( V(a \to b) + V(a \to d)) \geq m + 1$ (since as before $b \neq d$ implies $V(a \to b) \neq V(a \to d)$). So in order to fill in this diagram, we must have $dim ( V(a \to b) + V(a \to d)) = m + 1$. In this case $V(a \to b) + V(a \to d) = V(a \to c)$ and $R_{a,b,c,d} = 1$. Otherwise, $R_{a,b,c,d} = 0$. 

Now, suppose that $R_{a,b,c,d} = 1$. This means that $V(a \to b) + V(a \to d)$ has dimension $m + 1$. Since $dim(V(a \to b) + V(a \to d)) + dim (V(a \to b) \cap V(a \to d)) = dim(V(a\to d)) + dim(V(a \to b))$, we have that $dim(V(a \to b) \cap V(a \to d)) = m - 1$. We choose $e$ to be the unique vertex with edge $a \to e$ such that $V(a \to e) = V(a \to b) \cap V(a \to d)$. Thus we have cycles $a \to e \to b \to a$ and $a \to e \to d \to a$ in our graph. Now, notice that $R_{a,b,c,d} = 1$ also implies that we have cycles $a \to b \to c \to a$ and $a \to d \to c \to a$ in our graph. If we suppose $m< n- 1$, we can apply the tetrahedron property to the pair of cycles $$ \,\, a \to b \to c \to a, a \to e \to b \to a \text{ and } \,\, a \to d \to c \to a, a \to e \to d \to a $$ and use uniqueness of edges in our graph to see that we have an edge $e \to c$. So $e \to c \to b \to e$, $e \to c \to d \to e$ and $e \to a \to c \to e$ are cycles in our graph. The first two give us a consistent filling of the left hand side of the equation and so $L_{a,b,c,d} = 1$. 

If $m = n-1$, then $a = c$ and we can directly see that having $a \to b \to c \to a$ and $a \to e \to b \to a$ implies $e \to c \to b \to e$, and similarly for $d$, without having to use the tetrahedron property. 

Now, suppose that $L_{a,b,c,d} = 1$. This means that $dim(V(a \to b) \cap V(a \to d)) = m - 1$. So $dim(V(a \to b) + V(a \to d) ) = m + m - (m -1) = m +1$. Thus, we have that $V(a \to b) + V(a \to d)$ represents some edge $a \to g$ with label $m+1$, where we have cycles $a \to b \to g \to a$ and $a \to g \to d \to a$ in our graph. If $m < n-1$, we can apply the tetrahedron property to the pairs $$a \to b \to g \to a, \,\, a \to e \to b \to a \text{ and } a \to d \to g \to a, \,\, a \to e \to d \to a $$ to see that we have cycles $e \to b \to g \to e$, $a \to e \to g \to a$ and $e \to d \to g \to e$ in our graph as well. This shows us that we have a vertex $g$ in $G(\Delta)$ that is connected to $a$, $b$, $d$ and $e$. We also know that labels of the edges $b \to g$ and $d \to g$ are the same as the labels of $b \to c $ and $d \to c$. If the edge $a \to c$ exists, then it will also have the same label as $a \to g$. Also, recall that the labels of edges in any cycle must sum to $0 \mod n$. We know that $a \to e$ has label $m -1$ and $g \to a$ has label $n-(m+1)$, so $e \to g$ must be labeled with $2$. 

Now we need that $g = c$. To see this, we forget the subspaces with which we labelled edges starting at $a$. Instead, we apply remark \ref{link as subspaces} to $lk_{\Delta}(e)$ and label each edge with target vertex $x$ and label $i_x$ with a subspace $W(e \to x)$ of dimension $i_x$. So based on the the labeling of $lk_{\Delta}(e)$ of $e$, we have 1-dimensional subspaces $W(e \to b)$ and $W(e \to d)$. Since we have shown that $g$ is connected to both $b$ and $d$, the 2 dimensional subspace $W(e\to g)$ contains $W(e \to b) + W(e \to d)$. But since $b \neq d$, we have $W(e \to b) + W(e \to d)$ is 2-dimensional, so $W(e \to g) = W(e \to b) + W(e \to d)$. Note $c$ is also connected to both $b$ and $d$. We can repeat this reasoning with $W(e \to c)$ to see that $W(e \to c) = W(e \to b) + W(e \to d)$. So we must have $W(e \to c) = W(e \to g)$. Since each edge is labeled by a unique subspaces, that gives us $g = c$. This means that we have an edge $a \to c$ with label $m +1$ and therefore $R_{a,b,c,d} = 1$. 

If $m = n-1$, then $m + 1 = n$, so $a = c$. Thus the cycles on the right hand side are simply loops $a \to b \to a$ and $a \to d \to a$. 

Now, suppose $b = d$. We have that $$R_{a,b,c,b} =\begin{cases} 
   m & a \to b \to c \to a \text{ in } \Delta \\
   m-1 & \text{otherwise} 
  \end{cases}
$$
Now, we consider $L_{a,b,c,b}$, and see that we are counting vertices $e$ so that $a \to e \to b \to a$ and $e \to b \to c \to e$ are cycles in $\Delta$. To count these, we apply remark \ref{link as subspaces} to $lk_{\Delta}(b)$. That is, we define an $n-m$ dimensional subspace $S(b \to a)$, a $n - 1$ dimensional subspace $S(b \to e)$ and a 1 dimensional subspace $S(b\to c)$. Now, our previous statement about existence of cycles is equivalent to requiring $S(b \to c) \subset S(b \to e)$ and $S(b \to a) \subset S(b \to e)$. So we must choose for $S(b\to e)$ any $n -1$ dimensional subspace containing the subspace $S(b \to a) + S(b \to c)$. If $a \to b \to c \to a$ is a cycle in $\Delta$, we have $S(b\to c) \subset S(b \to a)$. So we are choosing an $n -1$ dimensional subspace containing a fixed $n -m$ dimensional subspace. By lemma \ref{Number of K-dim subspaces with fixed M dim subspace}, there are $\bigl[\!\begin{smallmatrix} m \\ m - 1\end{smallmatrix}\!\bigr]_q$ ways to to do this, which in our case ($q \equiv 1 \mod p$) reduces to $m$. If $a \to b \to c \to a$ is not a cycle, we are instead counting the number of $n -1 $ dimensional subspace containing a fixed $n - m + 1$ dimensional subspace. Again by lemma \ref{Number of K-dim subspaces with fixed M dim subspace}, there are $\bigl[\!\begin{smallmatrix} m - 1\\ m - 2 \end{smallmatrix}\!\bigr]_q$ ways to do this, which in our case reduces to $m -1$. In either case, we have $L_{a,b,c,b} = R_{a,b,c,b}$. 
\end{proof}

Since the square switch relations are the most complicated of the $Web(SL_n^{-})$ relations, this lemma almost completely gives us the following theorem. 

\begin{theorem}
\label{GPAFunctorThm}
Suppose $\Delta$ is an $\tilde{A}_{n-1}$ building of order $q$ and $G(\Delta)$ is the graph planar algebra of $\Delta$ over a field $\Bbbk$ of characteristic $p \geq n-1$ where $q \equiv 1 \mod p$. Then the maps we defined as the images of the $Web(SL^{-}_n)$ maps in the graph planar algebra satisfies the $Web(SL_n^{-})$ relations. Therefore these maps define a functor from $Web(SL_n^{-})$ to the graph planar algebra of $\Delta$ over $\Bbbk$. 
\end{theorem}
\begin{proof}
\textbf{Associativity and Co-associativity: } We must show that our maps satisfy relation \ref{associativity}. We do this by labeling the regions in the following ways 
$$
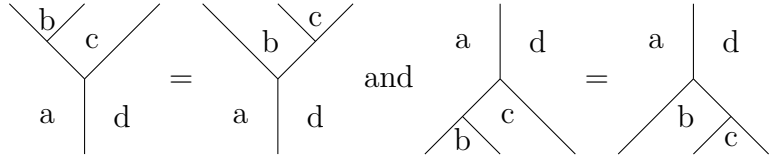
\begin{tikzpicture}[scale=1, baseline = -.1cm]
\draw(-1, 1) to (0,0);
\draw(1,1) to (0,0);
\draw(0,0) to (0, -1);
\draw (0, 1) to (-0.5, 0.5);
\node(none) at (-0.5,-0.5 ) {a};
\node(none) at (0.5,-0.5 ) {d};
\node(none) at (0.1, 0.5 ) {c};
\node(none) at (-0.5, 0.8) {b};
\end{tikzpicture} = \begin{tikzpicture}[scale=1, baseline = -.1cm]
\draw(-1, 1) to (0,0);
\draw(1,1) to (0,0);
\draw(0,0) to (0, -1);
\draw (0, 1) to (0.5, 0.5);
\node(none) at (-0.5,-0.5 ) {a};
\node(none) at (0.5,-0.5 ) {d};
\node(none) at (-0.1, 0.5 ) {b};
\node(none) at (0.5, 0.8) {c};
\end{tikzpicture} \text{ and } \begin{tikzpicture}[scale=1, baseline = -.1cm]
\draw(-1, -1) to (0,0);
\draw(1,-1) to (0,0);
\draw(0,0) to (0, 1);
\draw (0, -1) to (-0.5, -0.5);
\node(none) at (-0.5, 0.5 ) {a};
\node(none) at (0.5, 0.5 ) {d};
\node(none) at (0.1, -0.5 ) {c};
\node(none) at (-0.5, -0.8) {b};
\end{tikzpicture} = \begin{tikzpicture}[scale=1, baseline = -.1cm]
\draw(-1, -1) to (0,0);
\draw(1,-1) to (0,0);
\draw(0,0) to (0, 1);
\draw (0, -1) to (0.5, -0.5);
\node(none) at (-0.5,0.5 ) {a};
\node(none) at (0.5,0.5 ) {d};
\node(none) at (-0.1, -0.5 ) {b};
\node(none) at (0.5, -0.8) {c}; \end{tikzpicture} $$
See that the left hand side of these equations are 1 if $a \to b \to c \to a$ and $a \to c \to d \to a$ are cycles in our graph and that the right hand side of these equations are 1 if $a \to b \to d \to a$ and $b \to c \to d \to b$ are cycles in our graph. Now, note that these cycles are directly related by the tetrahedron property. So by lemma \ref{TetraProp}, both sides of this equation are 1 at exactly the same time (The edge sum condition on the tetrahedron property ensures that both sides of the equality are valid morphisms in $G(\Delta)$). 
\\
\\
\textbf{Bigon-Bursting: } We must show that our maps satisfy relation \ref{bigon bursting}. We do this by labeling the regions in the following ways 
$$\begin{tikzpicture}[scale=1, baseline = -.1cm]
\draw (0,-1)-- (0,-0.5) ;
\draw (270:0.5) arc (270:90:0.5);
\draw(-90:0.5) arc (-90:90:0.5);
\draw(0,0.5)--(0,1);
\node(none) at (0.8,0) {c};
\node(none) at (-0.8,0) {a};
\node(none) at (0,0) {b};
\end{tikzpicture} = {j +k \choose j} \begin{tikzpicture}[scale=1, baseline = -.1cm]
\draw (0,-1) to (0,1);
\node(none) at (0.3,0) {c};
\node(none) at (-0.3,0) {a};
\end{tikzpicture}$$
 We consider $lk_{\Delta}(a)$ and assign subspaces to each label. We see that $V(a \to b)$ is a $j$ dimensional subspace and that $V(a \to c)$ is a $j + k$ dimensional subspace. On the left hand side of this equation, we are counting the number of $b$'s so that $a \to b \to c \to a$ is a cycle in our graph. Notice that this is equivalent to counting the number of $j$ dimensional subspaces that are incident to the $j + k$ dimensional subspace $V(a \to c)$. We can do this by counting $n - j$ dimensional subspaces containing a fixed $n - (j +k)$ dimensional subspace. By lemma \ref{Number of K-dim subspaces with fixed M dim subspace}, there are $\bigl[\!\begin{smallmatrix} n - (n - (j+k)\\ n - j - (n - (j + k)) \end{smallmatrix}\!\bigr]_q = \bigl[\!\begin{smallmatrix} j + k\\ k \end{smallmatrix}\!\bigr]_q$. In our setting this equals ${j +k \choose k} = {j + k \choose j}$, as we have chosen $q \equiv 1 \mod p$. But this is exactly the value of the right hand side, since we have labelled this diagram in a unique way by definition. 

\textbf{Square Switch Relations:} Proposition 3.3 in \cite{Jones21} says that Lemma \ref{BuildingSS} gives us the general square switch relations whenever $1 \leq j,k \leq n-2$. So we must consider the cases where we have $j = n-1$ or $k = n-1$. We will show the case that $k = n -1$ for the first square switch relation and proofs of the other cases will follow similarly. 

In our setup, if no edges have labels $n$, $k = n-1$ then $m = n - 1$ and $\ell = 1$. So if we look at the possible values of $t$ for the right hand side of the equation, we see that the only possible value of $t$ giving a non-zero value is $j - 1$. So our relation reduces to the following diagram, 
$$\begin{tikzpicture}[scale=1, baseline = -.1cm]
\draw (-0.5, 1) to (0, 0.5);
\draw (0.5, 1) to (0, 0.5);
\draw (0, 0.5) to (0, 0.1);
\draw (-0.5, -1) to (0, -0.5);
\draw (0.5, -1) to (0, -0.5);
\draw (0, -0.5) to (0, -0.1);
 \mydotw{(90:0.1)};
 \mydotw{(90:-0.1)};
 \node(none) at (-0.5, -1.2) {n-1};
 \node(none) at (0.5, -1.2) {1};
 \node(none) at (-0.5, 1.2) {j};
 \node(none) at (0.5, 1.2) {n-j};
\end{tikzpicture} = \begin{tikzpicture}[scale=1, baseline = -.1cm]
\draw (-0.5, 1) to (0, 0.5);
\draw (0.5, 1) to (0, 0.5);
\draw (0, 0.5) to (0, 0.1);
\draw (-0.5, -1) to (0, -0.5);
\draw (0.5, -1) to (0, -0.5);
\draw (0, -0.5) to (0, -0.1);
 \mydotw{(90:0.1)};
 \mydotw{(90:-0.1)};
 \node(none) at (-0.5, -1.2) {n-1};
 \node(none) at (0.5, -1.2) {1};
 \node(none) at (-0.5, 1.2) {j};
 \node(none) at (0.5, 1.2) {n-j};
\end{tikzpicture}$$ 
which is trivial. The other cases similarly give trivial equalities of diagrams. So we see that the square switch relations hold in general. 

\textbf{$SL_n^{-}$ Relations: } We must show that our maps satisfy relation \ref{SLn relations}. We do this by labeling the regions in the following ways 
$$\begin{tikzpicture}[scale=1, baseline = -.1cm]
\draw (0, -1) to (0, 1);
\node(none) at (-0.3, 0) {a};
\node(none) at (0.3, 0) {b};
\end{tikzpicture} = \begin{tikzpicture}[scale=1, baseline = -.1cm]
\draw (-0.5, -1) to (-0.5, 0.8);
\draw (0.5, -0.8) to (0.5, 1);
\draw (-0.5, 0.4) to (0.5, -0.4);
\node(none) at (-0.7, 0) {a};
\node(none) at (0.7, 0) {b};
\node(none) at (0, 0.4) {a};
\node(none) at (0, -0.4) {b};
 \mydotw{(125:0.9)};
 \mydotw{(305:0.9)};
\end{tikzpicture} = \begin{tikzpicture}[scale=1, baseline = -.1cm]
\draw (-0.5, -0.8) to (-0.5, 1);
\draw (0.5, -1) to (0.5, 0.8);
\draw (-0.5, -0.4) to (0.5, 0.4);
\node(none) at (-0.7, 0) {a};
\node(none) at (0.7, 0) {b};
\node(none) at (0, 0.4) {b};
\node(none) at (0, -0.4) {a};
 \mydotw{(55:0.9)};
 \mydotw{(235:0.9)};
\end{tikzpicture}$$

The left hand side of this equation is 1 whenever there is an edge $a \to b$ of type $m$. Both the center and right hand side of this equality are 1 whenever there are loops $a \to b \to a$ and $b \to a \to b$, where $a \to b$ has type $m$ and $b \to a$ has type $n - m$. These conditions are exactly equivalent given the structure of our graph. 
\end{proof}
\section{Constructing Module Categories}
\label{Constructing Module Categories}

Mondule categories for a tensor category $\mathcal{C}$ are a categorification of the concept of $R$-modules for a ring $R$. Recall that we can make an Abelian group $A$ an $R$-module by choice of a ring homomorphism $R \to End(A)$. So for a tensor category $\mathcal{C}$, a $\mathcal{C}$-module category $\mathcal{M}$ can be characterized by a category $\mathcal{M}$ along with a monoidal functor $\mathcal{C} \to End(\mathcal{M})$. More details concerning module categories can be found in \cite[Section 7.1]{EGNO15} and definitions \ref{GroupActionCat}, \ref{EquivCat}, and \ref{EquivObject} are taken from \cite[Section 2.7]{EGNO15}. Our goal is to reinterpret the functor we established in the previous section as giving rise to a module category. We will then show that symmetries of the underlying building yield new module categories under the equivariantization construction. As a special case, we recover the main results of Jones \cite{Jones21} when the symmetries are vertex simply transitive. When $n = 2k+1$ and $\Bbbk$ is algebraically closed, these will also be $Tilt(SL_{2k+1})$ module categories. 
\subsection{Building a module category}
\label{extending functor}
We introduce a new category whose structure relies on an $\tilde{A}_{n-1}$ building $\Delta$. We will then establish a connection between this category and $G(\Delta)$ as in definition \ref{GPA}.
\begin{definition}
\label{VecDelta}
For a locally finite building $\Delta$ of type $\tilde{A}_{n-1}$ with set of 0-simplices (i.e. vertices) $V(\Delta)$ the category $Vec(\Delta)$ is defined as the category where: 
\begin{itemize}
\item Objects are tuples $V = (V_i)_{i\in V(\Delta)}$ of vector spaces.
\item Morphisms from $V$ to $W$ are collections of linear transformations $(f_i: V_i \to W_i)_{i \in V(\Delta)}$. 
\item Composition of morphisms is a component wise operation. 
\end{itemize}
\end{definition}

While $Vec(\Delta)$ ties a vector space to each vertex of our building (and thus of $\Gamma_\Delta$), it does not give us any information about which vertices are connected. We would like to use the structure of the 1-simplices of $\Delta$ (i.e. the edges of $\Gamma_\Delta$) to construct a endofunctor on $Vec(\Delta)$ in a way that encodes the data of $\Gamma_\Delta$. Consider the following collection of functors. 
\begin{definition}
\label{FmFunctor}
For $m \in \{0, ..., n-1\}$, define the functor $F_m: Vec(\Delta) \to Vec(\Delta)$ on objects as $(F_m(V))_{i} = \oplus_{k \in E(m, i)} V_{k}$, where $E(m, i) = \{ k \in V(\Delta) \,\,| \text { there exists an edge } i \to k \text { with label } m \text{ in } \Gamma_\Delta\}$ and on morphisms as $(F_m(f))_i = \oplus_{k \in E(m,i)} f_k$. 
\end{definition}

Consider the composition of functors of this form. For example $$(F_{m_1}F_{m_2}(V))_{i} = \oplus_{j \in E(m_1, i)} (F_{m_1}(V))_{j} = \oplus_{j \in E(m_1, i)} (\oplus_{k \in E(m_2, j)} V_{k}).$$ This double sum on the right hand of side of this equality can equivalently be indexed by all paths starting at $i$ whose edges have labels $m_1$ and $m_2$ respectively. So we see that the composition of these functors is encoding paths whose edge labels have a certain type. Recall that the type of a path $a_1 \to a_2 \to ... \to a_k$ is $\sigma = (\sigma_1, ..., \sigma_{k-1})$ where the edge with source $a_i$ has label $\sigma_i$. We will write $F_\sigma$ for the composition of functors $F_{\sigma_1}...F_{\sigma_{k-1}}$. 
\\
\\
Now, we will consider a special class of objects of $Vec(\Delta)$ which we will later show serves in some sense as a generating set for the entire category. 

\begin{definition}
\label{simpleobjects}
For each vertex $a \in V(\Delta)$ define an object $\Bbbk_a \in Vec(\Delta)$ where $(\Bbbk_a)_b = \Bbbk$ if $b = a$ and $(\Bbbk_a)_b = 0$ otherwise. 
\end{definition}

Now, consider $F_{\sigma}(V)$ for some path type $\sigma$. Notice that for arbitrary $a \in V(\Delta)$, $(F_{\sigma}(V))_a$ is a direct sum $\bigoplus V_b$ for all vertices $b$ that are a $\sigma$-type path away from $a$. This means that $(F_{\sigma}(V))_a \simeq \bigoplus_{b} (\bigoplus_{i = 0}^{dim(V_b)} \Bbbk)$ for all compatible $b$. We show $F_{\sigma}$ commutes with direct products. So we can simply move them inside the direct product and direct sums, and study $(F_{\sigma}(\Bbbk_a))_b$ instead.
\begin{lemma}
For any path type $\sigma$, $F_\sigma(\bigoplus_{i\in I} V^{i}) \cong \bigoplus_{i\in I} F_\sigma(V^{i})$ for all $V^{i}\in Ob(Vec(\Delta))$. 
\end{lemma}
\begin{proof}
We will show this isomorphism component-wise. For a sum $\bigoplus_{i\in I} V^{i}$ of objects in $V(\Delta)$ and a functor $F_m$ for $m \in \{0, ..., n-1\}$, 
$$(F_m(\bigoplus_{i \in I}V^{i}))_j = \bigoplus_{k \in E(m, j)}(\bigoplus_{i \in I} V^{i})_k = \bigoplus_{k \in E(m, j)}(\bigoplus_{i \in I} V^{i}_k) $$ We can then switch the order of sums to see that $$\bigoplus_{k \in E(m, j)}(\bigoplus_{i\in I} V^{i}_k) = \bigoplus_{i\in I}( \bigoplus_{k \in E(m, j)}V^{i}_k)) = \bigoplus_{i\in I} (F_m(V^{i}))_j$$
So we have the equality for each component of $F_m$ and therefore in general. 
\end{proof}
We will see that these $F_{\sigma}$ form a set of distinguished objects in $End(Vec(\Delta))$ that allow us to embed $G(\Delta)$ into $End(Vec(\Delta))$. To this end, we have the following theorem, which gives us a correspondence between morphisms in $G(\Delta)$ and $End(Vec(\Delta))$.
\begin{theorem}
\label{FmNatTransAreGPAMaps}
For functors $F_\sigma$ as defined in definition \ref{FmFunctor}, $$Nat(F_{\sigma}, F_{\tau}) \widetilde{=} Hom_{G(\Delta)}(\sigma, \tau)$$ as $\Bbbk$-vector spaces. 
\end{theorem}
\begin{proof}
Consider fixed path types $\sigma$ and $\tau$ and choose an arbitrary natural transformation $\eta: F_\sigma \rightarrow F_\tau$. Notice that $\eta$ is composed of maps $\eta_V$ for each $V \in Vec(\Delta)$. Each $\eta_V$ is further broken down into component maps $(\eta_V)_{a}$ for each vertex $a \in V(\Delta)$. By the previous paragraph, we need only consider the maps $(\eta_{\Bbbk_a})_b$ for vertices $a$ and $b$. We can then recover the overall structure of the transformation from these maps. 

First, we see that if $\sigma = (\sigma_1)$, then $(F_\sigma(\Bbbk_a))_b = \bigoplus_{c \in E(b, \sigma_1)} (\Bbbk_a)_c$. which is $\Bbbk$ if there exists an edge $b \to a$ of type $\sigma_1$ and 0 otherwise. So for a path type $\sigma$ of arbitrary length, we see that $(F_\sigma(\Bbbk_a))_b = \bigoplus_{p \in P(\sigma, b,a)} \Bbbk$, where we define $P(\sigma, b,a)$ as the set of paths of type $\sigma$ that start at $b$ and end at $a$. Notice that $\bigoplus_{p \in P(\sigma, b, a)} \Bbbk \simeq \Bbbk[P(\sigma, b,a)]$, the $\Bbbk$ vector space with $\sigma$ paths between $b$ and $a$ as a basis. So $(F_\sigma(\Bbbk_a))_b \simeq \Bbbk[P(\sigma, b,a)]$. Similarly, $(F_\tau(\Bbbk_a))_b \simeq \Bbbk[P(\tau, b,a)]$. So we can think of $(\eta_{\Bbbk_a}))_b$ as a map from the $ \Bbbk$-linear span of $P(\sigma, b, a)$ to the $\Bbbk$-linear span of $P(\tau, b,a)$. Since $\Delta$ is locally finite, these are both finite dimensional vector spaces. This yields a finite matrix $M = (m_{q,p})_{p \in P(\sigma, b,a), q \in P(\tau, b,a)}$ for $(\eta_{\Bbbk_a})_b$. For every pair $(p,q)$ we have a chosen element of $\Bbbk$, namely the matrix entry $m_{q,p}$. We can rearrange data we are given to form a linear functional from the $\Bbbk$-vector space spanned by the elements of $P(\sigma, b,a) \times P(\tau, b,a)$ to $\Bbbk$. 

Now, we collect all of the data from all choices of $b$ in $(\eta_{\Bbbk_a})_b$. So we can determine a functional from the space spanned by all matched pairs of paths with types $(\sigma, \tau)$ with final vertex $a$. Furthermore, the collection of maps $\eta_{\Bbbk_a}$ for all vertices $a \in V(\Delta)$ then encodes the data of a linear functional $f_{\sigma, \tau}$ from the space spanned by all matched paths of type $(\sigma, \tau)$ (call the matrix we get from this aggregation $M(\eta)$). This functional is uniquely determined by our choices of matrices $M$ for each pair of vertices $(a,b)$. On the other hand, for an arbitrary functional $f: \Bbbk[(p_1,p_2): type(p_1) = \sigma, type(p_2) = \tau] \rightarrow \Bbbk$, we can determine entries for the matrix $M$ for a given $a$ and $b$ by saying $M_{q, p} = f(qp)$. Now, $f$ is by definition in $Hom(\sigma, \tau) \in G(\Delta)$. So, this is an explicit construction of the isomorphism between $Hom_{G(\Delta)}(\sigma, \tau)$ and $Nat(F_\sigma, F_\tau)$ as desired. 
\end{proof}

Now, if $\eta \in Nat(F_{\sigma_1}, F_{\tau_1})$ and $\mu \in Nat(F_{\sigma_2}, F_{\tau_2})$, we consider the natural transformation $\eta \otimes \mu \in Nat(F_{\sigma_1}F_{\sigma_2}, F_{\tau_1}F_{\tau_2})$. We can describe a component of this transformation as the following composition
$$(F_{\sigma_1}F_{\sigma_2}(\Bbbk_m))_n \xrightarrow{(\eta_{F_{\sigma_2}(\Bbbk_m)})_n} (F_{\tau_1}F_{\sigma_2}(\Bbbk_m))_n \xrightarrow{F_{\tau_1}(\mu_{\Bbbk_m})_n} (F_{\tau_1}F_{\tau_2}(\Bbbk_m))_n.$$ 

We can identify $(F_{\sigma_1}F_{\sigma_2}(\Bbbk_m))_n$ with the $\Bbbk$-vector space of paths of type $\sigma_1\sigma_2$ (here listing path types consecutively means concatenation of lists) from $n$ to $m$, and similarly with $(F_{\tau_1}F_{\sigma_2}(\Bbbk_m))_n$ and $(F_{\tau_1}F_{\tau_2}(\Bbbk_m))_n$. So we can represent these maps by matrices as before, and their composition by matrix multiplication. Thus, we have matrices $M = M(\eta_{F_{\sigma_2}(\Bbbk_m)})_n)$ and $M' = M(F_{\tau_1}(\mu_{\Bbbk_m})_n)$. 

\begin{lemma}
\label{CompositionStructure}
For the matrices $M$ and $M'$ as defined above, we have the following 
$$M_{q'p', qp} = \begin{cases} 
   0 & p \neq p'\\
   M(\eta)_{q',q} & p = p'
\,\,\,\,\,\,\,
  \end{cases} M'_{q''p'', q'p'} = \begin{cases} 
   0 & q' \neq q''\\
   M(\mu)_{p'',p'} & q'' = q'\\
  \end{cases}$$
\end{lemma}
\begin{proof}
We will show this for a general object $X$ in $Vec(\Delta)$ and as a result will have the desired statement. We begin with showing the first claim. Define a map $e^{p_0,p_0'}: F_{\sigma_2}(X) \rightarrow F_{\sigma_2}(X)$ as the map from the vector space of $\sigma_2$ paths to itself sending the path $p_0$ to $p'_0$ and every other path to 0. Now, $F_{\sigma_1}(e_{p_0, p'_0})$ will act on $\sigma_1\sigma_2$ paths by fixing the $\sigma_1$ path and applying $e_{p_0, p_0'}$ to the $\sigma_2$ path. Let $E^{p_0 \rightarrow p'_0}$ be the matrix of $F_{\sigma_1}(e_{p_0, p_0'})$. So we have that $E^{p_0 \rightarrow p'_0}_{q_1p_1,q_1p_2} = \delta_{q_1 = q_2} \delta_{p_1 = p_0'} \delta_{p_2 = p_0}$. Since we will consider a fixed $p_0$ and $p_0'$ in the argument, take $E = E^{p_0\rightarrow p_0'}$. Now, we can repeat this process with $F_{\tau_1}(e_{p_0, p
_0})$ to get a matrix $\overline{E} = \overline{E}^{p_0 \rightarrow p_0'}$ for this map where $\overline{E}_{q'_1p_1, q'_2p_2} = \delta_{q'_1 = q'_2} \delta_{p'_1 = p_0'} \delta_{p'_2 = p_0}$.

Now, consider the following naturality diagram. 
$$\begin{tikzcd}
F_{\sigma_1}(F_{\sigma_{2}}(X)) \arrow[r, "F_{\sigma_1}(e_{p_0,p_0'})"] \arrow[d, "\eta_{F_{\sigma_{2}}(X)}"] & F_{\sigma_1}(F_{\sigma_2}(X)) \arrow[d, "\eta_{F_{\sigma_{2}}(X)}"] \\
F_{\tau_1}(F_{\sigma_{2}}(X)) \arrow[r, "F_{\tau_1}(e_{p_0,p_0'})"] & F_{\tau_1}(F_{\sigma_{2}}(X)) 
\end{tikzcd}$$ 
Commutivity of this diagram is simply saying that $\overline{E}M = ME$. This equality will put some restrictions on the entries of $M$. Now, we will expand the matrix multiplication at an entry and see that 
$$\overline{E}M_{q'_1p_1, q_2p_2} = \sum_{q'p \in P(\tau_1\sigma_2)} \overline{E}_{q'_1p_1, q'p}M_{q'p, q_2p_2} = \sum_{q'p \in P(\tau_1\sigma_2)} \delta_{q'_1 = q'} \delta_{p_1 = p_0'} \delta_{p = p_0} M_{q'p, q_2p_2} = M_{q_1'p_0, q_2p_2}$$ where we also carry the information that $p_1 = p_0'$. In all other cases, this matrix entry will be 0. Now, see also that 
$$ME_{q_1'p_1, q_2p_2} = \sum_{qp \in P(\sigma_1, \sigma_2)} M_{q_1'p_1, qp}E_{qp, q_2,p_2} =\sum_{qp \in P(\sigma_1, \sigma_2)} M_{q'_1p_1, qp} \delta_{q = q_2} \delta_{p = p_0'}\delta{p_2 = p_0} = M_{q_1'p_1, q_2p_0'} $$ where we carry the extra condition that $p_2 = p_0$. Now, we combine these expressions to see that $ M_{q_1'p_0, q_2p_2} = M_{q_1'p_1, q_2p_0'}$ where we must have $p_1 = p_0'$ and $p_2 = p_0$ for this entry to be non-zero. This last condition shows that our matrix entry can be rewritten as $M_{q'p_0, q_2p_0}$. So any non-zero entry of $M$ will have fixed $\sigma_2$ component. 

Now, we show this fact for $M'$. To see this, see that $F_{\tau_1}(\mu_{\Bbbk_m})_n : F_{\tau_1}(F_{\sigma_2}({\Bbbk_m}))_n \rightarrow F_{\tau_1}(F_{\tau_2}({\Bbbk_m}))_n$. Now, expanding this statement based on the action of $F_{\tau_1}$, we get this alternate form of the map. 
$$\bigoplus_{k \to n \in P(\tau_1)} (\mu_{\Bbbk_m})_k:\bigoplus_{k \to n \in P(\tau_1)} F_{\sigma_2}(\Bbbk_m)_k \rightarrow \bigoplus_{k \to n \in P(\tau_1)} F_{\tau_2}(\Bbbk_m)_k $$
Since this map is a compontent wise map, there is no way to map from one compentent (that is a fixed $\tau_1$ path) to another. So we must have that $M'$ has non-zero entries only when $q'' = q'$. 
\end{proof}
Now, consider the matrix $M'M$. See that $$M'M_{q''p'', qp} = \sum_{q'p' \in P(\tau_1,\sigma_2)} M'_{q''p''q'p'}M_{q'p', qp} =$$ $$ \sum_{q'p' \in P(\tau_1,\sigma_2)'} M(\mu)_{p'',p'}M(\eta)_{q'q} \delta_{p'= p} \delta_{q' = q''} = M(\mu)_{p'',p'}M(\eta)_{q'q}$$ since all other terms have either $p' \neq p$ or $q' \neq q''$. As before, these maps on components generalize to way to describe $\eta \otimes \mu$ as a matrix. Now, see that under the image of the isomorphism described earlier, $$f_{\eta} \otimes f_{\mu} (q_1p_1, q_2p_2) = f_\eta(q_1, q_2)f_\mu(p_1p_2)$$ by definition of the tensor product in $G(\Delta)$. So we will have $f_\eta \otimes f_\mu$ acts the same way as $f_{\eta \otimes \mu}$. 

In $End(Vec(\Delta))$, $F_{\sigma} \otimes F_{\tau} = F_{\sigma}F_{\tau}$ is a sum over paths of type $\sigma\tau$. Also, see that in $G(\Delta)$, $\sigma \otimes \tau$ is simply the concatenated sequence $(\sigma_1, ...., \sigma_n, \tau_1, ..., \tau_m)$ so that $F_{\sigma} \otimes F_{\tau} = F_{\sigma \otimes \tau}$. From this, we can extend theorem \ref{FmNatTransAreGPAMaps} to the following. 
\begin{theorem}
\label{GPAandEndVecAreME}
There exists a monoidal embedding $G(\Delta) \rightarrow End(Vec(\Delta))$.
\end{theorem}
This theorem means that when our categories are defined over a field $\Bbbk$ whose characteristic $p \geq n-1$ is such that $q \equiv 1 \mod p$ we can extend our functor $Web(SL_n^{-}) \rightarrow G(\Delta)$ as stated in section 5 to one from $Web(SL_n)^{-}$ to $End(Vec(\Delta))$. Theorems \ref{GPAFunctorThm} and \ref{GPAandEndVecAreME} together give us the proof of theorem A. 

\subsection{Equivariantization}

Now, suppose that $G$ is a group acting on $\Delta$ by permuting the vertices and preserving adjancency. We can use the action of $G$ to create other interesting $Web(SL_n^{-})$-module categories. Recall $Cat(G)$ is the monoidal category where objects are group elements, the only morphisms are identities and the tensor product is group multiplication. Also, recall $Aut(\mathcal{C})$ is the category of auto-equivalences on $\mathcal{C}$ where morphisms are natural isomorphisms of functors (See \cite[Section 2.7]{EGNO15} for these definitions). 
\begin{definition}
\label{GroupActionCat}
Given a group $G$ and a category $\mathcal{C}$, an action of $G$ on $\mathcal{C}$ is a monoidal functor $A: Cat(G) \rightarrow Aut(C)$. We denote this $g \mapsto A_g \in Aut(\mathcal{C})$, and the ``tensorator" isomorphisms $\gamma_{g,h}: A_g \circ A_h \simeq A_{gh}$ for all $g,h \in G$. 
\end{definition}
\begin{definition}
\label{EquivObject}
An $G$-equivariant object $(X, u)$ in $\mathcal{C}$ is $X \in Ob(\mathcal{C})$ together with a choice of a family of isomorphisms $u = (u_g: A_g(X) \to X )_{g \in G}$ such that the following diagram commutes. 
$$\begin{tikzcd}
A_g(A_h(X)) \arrow[r, "A_g(u_h)"] \arrow[d, "\gamma_{g,h}"] & A_g(X) \arrow[d, "u_g"] \\
A_{gh}(X) \arrow[r, "u_{gh}"] & X
\end{tikzcd}$$ 
\end{definition}

\begin{definition} 
\label{EquivCat}
The collection of $G$-equivariant objects of $\mathcal{C}$ forms a category $\mathcal{C}^{G}$, called the $G$-equivariantization of $\mathcal{C}$. Morphisms in $Hom_{\mathcal{C}^{G}}((X,u), (Y,v))$ are simply morphisms $f$ in $Hom_\mathcal{C}(X,Y)$ such that the following diagram commutes. 
$$\begin{tikzcd}
A_g(X) \arrow[r, "A_g(f)"] \arrow[d, "u_g"] & A_g(Y) \arrow[d, "v_g"] \\
X\arrow[r, "f"] & Y
\end{tikzcd}$$ 
\end{definition}
Our goal is to extend the functor constructed in Theorem \ref{GPAandEndVecAreME} to one from $Web(SL_n^{-})$ to $End(Vec(\Delta)^{G})$ and then explore the effects of different $G$ actions on $\Delta$. As we begin, note that an element $g\in G$ acts on $\Delta$ by permuting the vertices. Thus, $(A_g(V))_i = V_{g^{-1}(i)}$ for every vertex $i \in V(\Delta)$. As $A_g(A_h(V))$ and $A_{gh}(V)$, we can choose $\gamma_{g,h}$ to be the identity maps. Also, recall that a type rotating action of $G$ on $\Delta$ is one where $type(g(i)) = type(i) + c \mod n$ for some fixed $c$ and every $i \in V(\Delta)$, $g \in G$. Our next task is to find an analogue of our functors $F_a$ from definition \ref{FmFunctor} on $Vec(\Delta)^{G}$. We will make use of the following lemma in our discussion. 

\begin{lemma}
\label{CommutingFunctors}
If $G$ acts on $\Delta$ by a type rotating action, then the functors $A_g$ that come from the corresponding action on $Vec(\Delta)$ commute with the functors $F_m$ as defined in \ref{FmFunctors} for all $g \in G$, $m \in \{0, ..., n-1\}$. That is, $F_mA_g = A_gF_m$ for all $m \in \{0, ..., n-1\}$ and $g \in G$. 
\end{lemma}
\begin{proof} We will show that $F_m(A_g(V))_i = A_g(F_m(V))_i$ for all $V \in Vec(\Delta)$ and $i \in V(\Delta)$. To do this, see that 
$$F_m(A_g(V))_i = \bigoplus_{k \in E(m, i)} (A_g(V))_k = \bigoplus_{k \in E(m, i)} V_{g^{-1}(k)}$$ and that $$A_g(F_m(V))_i = F_m(V)_{g^{-1}(i)} = \bigoplus_{k \in E(m, g^{-1}(i))} V_k$$
we have equality between these two expressions whenever $k \in E(m, i)$ implies that $g^{-1}(k) \in E(m, g^{-1}(i))$. But of course, this occurs for any type-rotating action of $G$. So we have that these two classes of functors commute. 
\end{proof}

Now, we turn our attention to the category $Vec(\Delta)^{G}$ and make the following definition, which will allow us to begin to extend our result in theorem \ref{GPAandEndVecAreME} to this category. 
\begin{definition}
\label{FmGFunctors}
For a given $m \in \{0,...,n-1\}$, define a functor $F_m^{G} \in End(Vec(\Delta)^{G})$ as acting on a object $(X,u)$ in the following way, $F_m^{G}(X,u) = (F_m(X), F_m(u))$, where $F_m(X)$ is as defined in \ref{FmFunctors} and $F_m(u)_g = F_m(u_g)$.
\end{definition}
We arrive at this definition for $F_m^{G}(u)$ in the following manner. By lemma \ref{CommutingFunctors}, we have that $A_g(F_m(X)) = F_m(A_g(X))$. So we have the following diagram 
$$\begin{tikzcd}
A_g(F_m(X)) \arrow[r, "F_m(u)_g"] \arrow[d, "="] & F_m(X)\\
F_m(A_g(X)) \arrow[ru, swap, "F_m(u_g)"]& 
\end{tikzcd}$$
This diagram shows us that the definition we made is the obvious choice for $F_m(u)$. Now we must verify that $F_m^{G}(X,u)$ is indeed an equivariant object. Recall that this means that we must show that the following diagram commutes.
$$\begin{tikzcd}[row sep=large, column sep=large]
A_g(A_h(F_m(X)) \arrow[r, "A_g(F_m(u)_h)"] \arrow[d, "="] & A_g(F_m(X)) \arrow[d, "F_m(u)_g"]\\
A_{gh}(F_m(X)) \arrow[r, "F_m(u)_{gh}"] & F_m(X)
\end{tikzcd}$$
To do this, we expand the diagram above as follows
{\normalfont \[\begin{tikzcd}
	& {} \\	
	\\
	{A_g(A_h(F_m(X))} &&&& {A_g(F_m(X))} \\
	& {A_g(F_m(A_h(X))} \\
	& {F_m(A_g(A_h(X))} && {F_m(A_g(X))} \\
	& {F_m(A_{gh}(X))} \\
	{A_{gh}(F_m(X))} &&&& {F_m(X)}
	\arrow["{A_g(F_m(u)_h)}", from=3-1, to=3-5]
	\arrow["{=}"', from=3-1, to=4-2]
	\arrow["{=}"', from=4-2, to=5-2]
	\arrow["{=}"', from=5-2, to=6-2]
	\arrow["{=}"', from=3-1, to=7-1]
	\arrow["{F_m(u)_g}"', from=3-5, to=7-5]
	\arrow["{F_m(u)_{gh}}"', from=7-1, to=7-5]
	\arrow["{=}"', from=7-1, to=6-2]
	\arrow["{F_m(u_{gh})}", from=6-2, to=7-5]
	\arrow["{F_m(A_g(u_h))}"', from=5-2, to=5-4]
	\arrow["{A_g(F_m(u_h))}"{description}, from=4-2, to=3-5]
	\arrow["{=}"{description}, from=3-5, to=5-4]
	\arrow["{F_m(u_g)}"{description}, from=5-4, to=7-5]
\end{tikzcd}\] }
We see that the pentagon on the left commutes because of equality. All three triangles and the top quadrilateral commute by our definition of $F_m(u)$. The bottom center quadrilateral commutes since $X$ is an equivariant object and this square is simply $F_m$ applied to the square in the definition \ref{EquivCat} and the top center quadrilateral commutes by Lemma \ref{CommutingFunctors}. So $F_m^{G}(X, u)$ is indeed an equivariant object as desired and $F_m^{G}$ for $m \in \{0,...,n-1\}$ gives us a collection of endofunctors on $Vec(\Delta)^{G}$. 

Now, equipped with our functors, we turn to the study of natural transformations between these functors. We want to show that a subcategory of $G(\Delta)$ can be mapped to these functors and thus embedded into $End(Vect(\Delta)^{G})$. Then we will show that the image of our $Web(SL_n^{-})$ maps are contained in this subcategory, and so this correspondence will allow us to build a functor from $Web(SL_n^{-})$ to $End(Vec(\Delta)^{G})$. 

Recall that the motivation for theorem \ref{GPAandEndVecAreME} was the isomorphism between $Nat(F_\sigma, F_\tau)$ and $Hom_{G(\Delta)}(\sigma, \tau)$. So we wish to find some condition on natural transformations in $Vec(\Delta)^{G}$ that allows us to construct a similar isomorphism between a subclass of morphisms in $G(\Delta)$ and these distinguished natural transformations. To do this, we must define what the $G$-action does to a natural transformation $\eta \in Nat(F_\sigma, F_\tau)$.
\begin{definition}
\label{GNatTrans}
For a natural transformation $\eta \in Nat(F_\sigma, F_\tau)$, define $g(\eta) \in Nat(F_\sigma^{G}, F_\tau^{G})$ as the natural transformation with components $g(\eta)_{(X, u)} = F_\tau(u_g)A_g(\eta_X)F_\sigma(u_g)^{-1}$.
\end{definition}

\begin{lemma}
\label{GInvariantMorphismsAreEquivNatTrans} If $\eta \in Nat(F_\sigma, F_\tau)$ is invariant under the action of $G$ on $Vec(\Delta)$ (that is, the component maps $\eta_X$ and $g(\eta)_{(X,u)}$ are equal for any choice of $u$), then $\eta$ induces a natural transformation $\eta^{G} \in Nat(F_\sigma^{G}, F_\tau^{G})$.
\end{lemma}
\begin{proof}
Suppose we have $\eta \in Nat(F_\sigma, F_\tau)$ which is invariant under the action of $G$. This means that for every $g \in G$ and $X \in Vec(\Delta)$, we have that $\eta_X = F_\tau(u_g)A_g(\eta_X)F_\sigma(u_g)^{-1}$. This equality gives us that the following diagram commutes. 
$$\begin{tikzcd}
A_g(F_\sigma(X)) \arrow[r, "A_g(\eta_X)"] \arrow[d, "F_\sigma(u_g)"] & A_g(F_\tau(X)) \arrow[d, "F_\tau(u_g)"] \\
F_\sigma(X)\arrow[r, "\eta_X"] & F_\tau(X)
\end{tikzcd}$$ 
This diagram is exactly the diagram we need to show that $(\eta_X)_{X\in Vec(\Delta)}$ is a collection of equivariant morphsims. Naturally of these maps follows from naturality of $\eta$, so we have for each $(X,u)$ a map $\eta^{G}_{(X,u)}$, where the collection $\eta^{G} \in Nat(F_\sigma^{G}, F_\tau^{G})$. 
\end{proof}
Now, if $\eta$ is invariant under the action of $G$, we will have that $M(\eta)_{p,q} = M(\eta)_{g(p),g(q)}$ for all paths $(p,q)$ of appropriate type. This means in our theorem \ref{GPAandEndVecAreME} correspondence, the map $f \in Hom_{G(\Delta)}(\sigma, \tau)$ is invariant under the action of $G$ on $\Delta$. In fact, we can also reverse this statement. If $f \in Hom_{G(\Delta)}(\sigma, \tau)$ is invariant under the action of $G$ (that is $f(pq) = f(g(p)g(q))$ for any $g\in G$ and $p, q$ suitable paths), then the natural transformation it is mapped to also has this property and so carries an induced $Vec(\Delta)^{G}$ natural transformation.

\begin{lemma} The image of any web in $Web(SL_n^{-})$ under the functor defined in theorem \ref{GPAFunctorThm} is invariant under any type-rotating action of $G$ on $\Delta$. 
\end{lemma}
\begin{proof}
Recall that the images of the $Web(SL_n^{-})$ webs in $G(\Delta)$ are morphisms whose co-domain is simply $\{0,1\}$. So we must only show that the existence of the left triangle implies that of the right triangle. 
$$\begin{tikzcd}
a \arrow[r, "j"] \arrow[rd, "j+ k"] & b \arrow[d, "k"] \\
 & c
\end{tikzcd} \,\,\,\,\,\,\,\,\,\,\,\\,\begin{tikzcd}
g(a) \arrow[r, "j"] \arrow[rd, "j+ k"] & g(b) \arrow[d, "k"] \\
 & g(c)
\end{tikzcd}$$ 
Since the action of $G$ is a graph automorphism, $g(a)$, $g(b)$ and $g(c)$ will be connected. Now, since the action of $G$ is type-rotating, we have that $type(a) - type(b) = type(g(a)) - type(g(b))$ and so all edges will have the same labels as well. So $\mathbbm{1}_{j,k}((a \to b \to c, a \to c)) = \mathbbm{1}_{j,k}((g(a) \to g(b) \to g(c), g(a) \to g(c))). $ The argument that $\mathbbm{1}_{j+k}$ and the other generating webs are invariant follows similarly. 
\end{proof}
The previous two lemmas allow us to extend theorem \ref{GPAandEndVecAreME} to the category $Web(SL_n^{-})$. This then gives us the following theorem. 
\begin{theorem}
\label{EquivMoniodalCategory}
For a fixed $n$, if $\Bbbk$ is a field of characteristic $p \geq n - 1$ and $\Delta$ is an $\tilde{A}_{n-1}$ building of order $q \equiv 1 \mod p$, $Vec(\Delta)^{G}$ has the structure of a $Web(SL_n^{-})$-module category, where the action is by the equivalence in theorem \ref{GPAandEndVecAreME} pre-composed with the functor defined in theorem \ref{GPAFunctorThm}.  
\end{theorem}

This theorem, along with theorem $\ref{GPAFunctorThm}$ is enough to prove theorem B. Now, we have a general action of $Web(SL_n^{-})$ on $Vec(\Delta)^{G}$, where the specific structure of $Vec(\Delta)^{G}$ is dependent on the action of $G$. We can consider certain classes of actions and what the functor $Web(Sl_n^{-}) \rightarrow Vec(\Delta)^{G}$ looks like under these classes.
\begin{example}
\normalfont
\label{TPModuleCategory}
Recall that a simply transitive action of $G$ on $\Delta$ is one for which there exists a unique $g \in G$ where $g(x) = y$ for every $x,y \in V(\Delta)$. \cite{CMSZ93} and \cite{Cart95} show that when $G$ acts simply transitively on $\Delta$, we can reduce the combinatorial information in $\Delta$ to a simpler object called a triangle presentation of type $\tilde{A}_{n-1}$. \cite{Jones21} further showed that there exists a fiber functor from $Web(SL_n^{-})$ that in certain setting extends to a fiber functor on $Tilt(SL_{2k+1})$ whose structure comes from these triangle presentations. Now, when $G$ admits a simply transitive action on $\Delta$, that $Vec(\Delta)^{G} \simeq Vec$. So we recover the existence of the previously discovered fiber functor as a special case of theorem 7.3. 
\end{example}

\begin{example}
\normalfont
\label{TransitiveMCGeneral}
In general, we have for any transitive action of $G$ on $\Delta$, that $Vec(\Delta)^{G} \simeq Rep(Stab(*))$, where $*$ is any vertex in $\Delta$. So, if we consider the stabilizers of various actions, we can realize some representation categories as $Web(SL_n^{-})$ module categories. 
\end{example}

\begin{example}
\label{PGLnMCEx}
\normalfont
A natural place to start is to consider the action of $PGL_n(K)$ on the building described in example \ref{BuidlingExample}. See that under this action, the stabilizer of the standard lattice ($Ae_1 \oplus ... \oplus Ae_n$ where $e_1, ..., e_n$ is the standard basis for $K^{n}$) is $SL_n(A)$. So, through this action, we have that $Rep(SL_n(A))$ is a $Web(SL_n^{-})$ module category. 

\end{example}

\bibliographystyle{alpha}
\bibliography{refs}

\end{document}